\documentclass[12pt]{article}

\usepackage[a4paper,margin=1in]{geometry}
\usepackage{amsmath}

\usepackage{graphicx, enumitem}
\usepackage{amssymb}
\usepackage{amsmath}
\usepackage{verbatim,booktabs}
\usepackage{subfigure,epsfig,url,psfrag}
\usepackage{float}
\usepackage{mathrsfs}

 \usepackage[usenames,dvipsnames]{pstricks}
 \usepackage{epsfig}
 \usepackage{pst-grad} 
 \usepackage{pst-plot} 

\makeatletter
\newcommand{\setwindow}[5]{
\def\xmin{#1}%
\def\ymin{#2}%
\def\xmax{#3}%
\def\ymax{#4}%
\pstFPsub\viewingwidth{#3}{#1}%
\pstFPdiv\result{\strip@pt#5}{\viewingwidth}%
\psset{unit=\result pt}}
\makeatother

\newcommand{\din}{d^{\text{in}}}
\newcommand{\dout}{d^{\text{out}}}

\usepackage{amsthm}
\newtheorem{theorem}{Theorem}
\newtheorem{corollary}{Corollary}
\newtheorem{proposition}{Proposition}
\newtheorem{lemma}{Lemma}

\newtheorem{remark}{Remark}

\def\dim{\mathop{\rm dim}}

\def\R{{\mathbb R}}

\def\ie{{i.e.,} }
\def\eg{{e.g. }}

\def\S{{\mathcal S}}
\def\H{{\mathcal H}}
\def\C{{\mathcal C}}

\def\B{{\mathcal B}}
\def\K{{\mathcal K}}

\def\Z{{\mathcal Z}}

\def\01{\ensuremath{0\mathord{-}1}}

\newcommand{\Tr}{\text{Tr}}

\newcommand{\bi}{\begin{list}{$\bullet$}{\setlength{\parsep}{0pt}\setlength{\itemsep}{0pt}}}

\newcommand\res{\mathop{\hbox{\vrule height 7pt width .3pt depth 0pt
\vrule height .3pt width 5pt depth 0pt}}\nolimits}

\newcommand{\norm}[1]{\Big\lVert#1\Big\rVert}

\newcounter{claim} 

\newcounter{mynotes}
\usepackage[ruled,vlined]{algorithm2e}

\def\Xint#1{\mathchoice
{\XXint\displaystyle\textstyle{#1}}%
{\XXint\textstyle\scriptstyle{#1}}%
{\XXint\scriptstyle\scriptscriptstyle{#1}}%
{\XXint\scriptscriptstyle\scriptscriptstyle{#1}}%
\!\int}
\def\XXint#1#2#3{{\setbox0=\hbox{$#1{#2#3}{\int}$ }
\vcenter{\hbox{$#2#3$ }}\kern-.6\wd0}}

\def\dashint{\Xint-}
\def\Xsum#1{\mathchoice
{\XXsum\displaystyle\textstyle{#1}}%
{\XXsum\textstyle\scriptstyle{#1}}%
{\XXsum\scriptstyle\scriptscriptstyle{#1}}%
{\XXsum\scriptscriptstyle\scriptscriptstyle{#1}}%
\!\sum}
\def\XXsum#1#2#3{{\setbox0=\hbox{$#1{#2#3}{\sum}$ }
\vcenter{\hbox{$#2#3$ }}\kern-.5\wd0}}

\def\dashsum{\Xsum-}

\DeclareMathOperator{\prob}{\mathbb P}
\DeclareMathOperator{\avg}{\mathbb E}

\title{The ratio-cut polytope and K-means clustering
\thanks{A. De Rosa was supported in part by National Science Foundation award DMS-1906451-2112311.}}

\author{Antonio De Rosa
\thanks{Department of Mathematics, University of Maryland, 4176 Campus Dr, College Park, MD 20742, USA.
Department of Decision Sciences and BIDSA, Bocconi University, Via R\"ontgen 1, 20136, Milano, MI, Italy.
             E-mail: {\tt  anderosa@umd.edu}.
             }
\and
Aida Khajavirad
\thanks{Department of Industrial and Systems Engineering, Lehigh University, Bethlehem, PA 18015, USA.
             E-mail: {\tt aida@lehigh.edu}.
             }
}
\date{}

\begin{document}

\maketitle

\begin{abstract}
We introduce the ratio-cut polytope defined as the convex hull of ratio-cut vectors corresponding to all partitions of $n$ points
in $\R^m$ into at most $K$ clusters. This polytope is closely related to the convex hull of the feasible region of
a number of clustering problems such as
K-means clustering and spectral clustering. We study the facial structure of the ratio-cut polytope and derive several types of facet-defining inequalities. We then consider the problem of K-means clustering and introduce a novel linear programming (LP) relaxation for it. Subsequently,
we focus on the case of two clusters and derive a sufficient condition under which the proposed LP relaxation recovers the underlying clusters exactly. Namely, we consider the stochastic ball model, a popular generative model for K-means clustering, and we show that if the separation distance between cluster centers satisfies $\Delta > 1+\sqrt 3$, then the LP relaxation recovers the planted clusters with high probability. This is a major improvement over the only existing recovery guarantee for an LP relaxation of K-means clustering stating that recovery is possible with high probability if and only if $\Delta > 4$.
Our numerical experiments indicate that the proposed LP relaxation significantly outperforms a popular semidefinite programming relaxation in recovering the planted clusters.
\end{abstract}


\vspace{0.1cm}


\section{Introduction}
Clustering is concerned with partitioning a given
set of data points $\{x^i\}_{i=1}^n$ in $\R^m$ into $K$ subsets such that some dissimilarity function
among the points is minimized. Consider a partition of $[n] := \{1,\ldots, n\}$; \ie $\{\Gamma_k\}_{k =1}^K$ such that
$\Gamma_a \cap \Gamma_b = \emptyset$ for all $a, b \in [K] := \{1,\ldots, K\}$ and $\cup_{k \in [K]} {\Gamma_k} = [n]$, where we further assume
$\Gamma_k \neq \emptyset$ for all $k \in [K]$. \emph{K-means clustering} partitions the data points into $K$ clusters by minimizing the total squared distance between each data point and the corresponding cluster center:
\begin{align}
\label{Kmeans1}
\min  \quad & \sum_{k=1}^K{\sum_{i \in \Gamma_k} { \norm{x^i - \frac{1}{|\Gamma_k|}\sum_{j \in \Gamma_k}{x^j}}_2^2}}\\
\text{s.t.} \quad & \{\Gamma_k\}_{k \in [K]} \; {\rm is \; a \; partition \; of \; [n]}.  \nonumber
\end{align}
It is well-known that Problem~\eqref{Kmeans1} is NP-hard even when there are only two clusters~\cite{BodGeo19}
or when the data points are in $\R^2$~\cite{MahNimVar09}.
The most famous heuristic for K-means clustering is Lloyd’s algorithm~\cite{Lloyd82}
which, in spite of its effectiveness, in practice
may converge to a local minimum that is arbitrarily bad compared to the global minimum~\cite{Kan02}.
Moreover,  numerous constant-factor approximation algorithms have been developed, both for the fixed number of clusters $K$ and for the fixed dimension $m$ (see for example~\cite{Kan02,FriRezSal19}). In this paper, we are interested in the quality of convex relaxations for K-means clustering.

Several equivalent reformulations of K-means clustering, including a nonlinear binary program~\cite{Rao71},
a binary semidefinite program (SDP)~\cite{PenWei07}, and a completely positive program~\cite{PraHan18} are given in the literature.
In the following, we present an alternative formulation that we will use to construct our new convex relaxation
(see~\cite{LiLiLi20} for the derivation).
Consider partition of $[n]$; let ${\bf 1}_{\Gamma_k}$, $k \in [K]$ be the indicator vector of
the $k$th cluster; \ie  the $i$th component of ${\bf 1}_{\Gamma_k}$ is defined as: $({\bf 1}_{\Gamma_k})_i = 1$ if $i \in \Gamma_k$ and $({\bf 1}_{\Gamma_k})_i = 0$ otherwise. Define
the \emph{partition matrix} as
\begin{equation}\label{pm}
Z = \sum_{k=1}^K{\frac{1}{|\Gamma_k|}{\bf 1}_{\Gamma_k}{\bf 1}^T_{\Gamma_k}}.
\end{equation}
Denote by $D \in \R^{n \times n}$ the distance matrix with the $(i,j)$th entry given by $d_{ij} = ||x^i - x^j||_2^2$.
Then it can be shown that Problem~\eqref{Kmeans1} can be equivalently written as:
\begin{align}
\label{Kmeans2}
\min \quad & \sum_{i,j \in [n]} {d_{ij} Z_{ij}}\\
\text{s.t.} \quad & Z \; {\rm is \; a \; partition\; matrix \; defined \; by~\eqref{pm}}.  \nonumber
\end{align}
\paragraph{\textbf{SDP relaxations.}}
The most popular convex relaxations for K-means clustering are SDP relaxations; indeed, both theoretical and numerical properties
of these algorithms have been extensively investigated in the literature (see for example~\cite{PenXia05,PenWei07,PraHan18}).
These relaxations are obtained by observing that any partition matrix $Z$ satisfies the following properties:
\begin{align}
& Z {\bf 1}_n = {\bf 1}_n, \quad {\rm Tr}(Z) = K\label{sdpR}, \\
& Z \succeq 0, \quad Z \geq 0 \nonumber,
\end{align}
where ${\bf 1}_n$ is a $n$-vector with all entries equal to $1$ and $ {\rm Tr}(Z)$ is the trace of the matrix $Z$. Moreover,
$Z \succeq 0$ and $Z \geq 0$ mean that the matrix $Z$ is positive semidefinite and component-wise nonnegative, respectively. A widely-studied SDP relaxation of the K-means clustering, often referred to as ``Peng-Wei relaxation''~\cite{PenWei07}, is given by
\begin{align}\label{SDP}
\min \quad & \sum_{i,j \in [n]} {d_{ij} Z_{ij}}\\
\text{s.t.} \quad & Z {\bf 1}_n = {\bf 1}_n, \quad {\rm Tr}(Z) = K,  \quad Z \succeq 0, \quad Z \geq 0,\quad Z = Z^T  \nonumber.
\end{align}
If by solving Problem~\eqref{SDP}, we obtain a minimizer $\bar Z$ that is a partition matrix as defined by~\eqref{pm}, then $\bar Z$ is also optimal for the original problem, as the feasible region of Problem~\eqref{SDP} contains the feasible region of Problem~\eqref{Kmeans2}. Otherwise, a common approach is to devise a \emph{rounding scheme} to extract a feasible solution of~\eqref{Kmeans2} from the relaxation solution $\bar Z$. This two-phase approach has been successfully employed for clustering various synthetic and real data sets~\cite{PenWei07,PraHan18}.

\paragraph{\textbf{Recovery guarantees under stochastic models.}}
A recent stream of research in data clustering is concerned with obtaining conditions under which a planted clustering corresponds to the unique optimal solution of a SDP relaxation under suitable generative models~\cite{Awaetal15,AbbBanHal16,HajWuXu16,MixVilWar17,IduMixPetVil17,Ban18,AgaBanKoiKol18, LinStr19,LiLiLi20}.
Such conditions are often referred to as~\emph{exact recovery} (henceforth, simply \emph{recovery}) conditions.
Generally speaking, these works first provide deterministic sufficient conditions for a given clustering assignment to be the unique optimal solution of a SDP relaxation via the construction of dual certificates.
Subsequently, they show that those conditions hold with high probability under a given random model.
Throughout this paper, we say that an optimization problem~\emph{recovers} the planted clusters
if its unique optimal solution corresponds to the planted clusters.

The stochastic ball model (SBM) is the most widely-studied generative model for K-means clustering. In this distributional setting,
we assume that there exist $K$ clusters in $\R^m$ and the data in
each cluster consists of $\frac{n}{K}$ points sampled from a uniform distribution within a ball of unit radius.
The question is what is the \emph{minimum separation distance} $\Delta$ between cluster centers needed for a convex relaxation to recover
these $K$ clusters with high probability (\ie probability tending to 1 as $n \rightarrow \infty$).
Clearly, a convex relaxation recovers the planted clusters
only if the original K-means problem succeeds in doing so. Perhaps surprisingly, the recovery threshold for K-means clustering
under the SBM remains an open question.
Recently, in~\cite{BodGeo19}, the authors prove that when the points are generated uniformly on two $m$-dimensional touching spheres for $m \geq 3$,
in the limit, \ie when the empirical samples is replaced by the probability measure,
K-means clustering identifies the two individual spheres as the two clusters.
In this paper, we show that the same recovery result is valid for the SBM.

In~\cite{Awaetal15}, the authors
consider the Peng-Wei relaxation as defined by~\eqref{SDP} and show that if $\Delta > 2 \sqrt{2} (1+1/\sqrt{m})$, then the SDP recovers the planted clusters with high probability.
In~\cite{IduMixPetVil17}, the authors consider the same SDP and prove that recovery is guaranteed with high probability if
$\Delta > 2+K^2/m$, which is near optimal in $m \gg K^2$ regime. The authors of~\cite{LiLiLi20}
obtain yet another recovery condition for Peng-Wei relaxation given by $\Delta > 2+O(\sqrt{K/m})$ which is an improvement over the previous condition when $K$ is large. Moreover, in~\cite{LiLiLi20}, the authors prove that if
$\Delta < 1+\sqrt{1+2/(m+2)}$, then with high probability, Peng-Wei relaxation fails in recovering the planted clusters.

\paragraph{LP relaxations.} It is widely accepted that state-of-the-art LP solvers outperform the SDP counterparts
in both speed and scalability. However, for K-means clustering,  to date, there exists no LP relaxation with desirable theoretical or
computational properties. In~\cite{Awaetal15}, the authors consider the following LP relaxation of K-means clustering:
\begin{align}
\label{badLP}
\min \quad & \sum_{i,j \in [n]} {d_{ij} Z_{ij}}\\
\text{s.t.} \quad & Z {\bf 1}_n = {\bf 1}_n, \quad {\rm Tr}(Z) = K, \quad Z \geq 0, \quad Z = Z^T,  \quad Z_{ij} \leq Z_{ii}, \quad \forall i, j \in [n]\nonumber.
\end{align}
Subsequently, they show that under the SBM, Problem~\eqref{badLP} recovers the planted clusters with high probability
if and only if $\Delta > 4$.
We should remark that  if $\Delta > 4$, any two points within a particular cluster are closer to each other than any two points from
different clusters, and hence in this case, recovery follows from a simple distance thresholding. They complement this negative theoretical result
with poor numerical performance to conclude the ineffectiveness of the ``natural'' LP relaxation for K-means clustering.
In~\cite{dPIdaTim20,AidaTonio21}, the authors study the recovery properties of LP relaxations for community detection and joint object matching.

\paragraph{\textbf{Our contribution.}}
In this paper, we propose a novel LP relaxation for K-means clustering with favorable theoretical and numerical properties.
We start by introducing the~\emph{ratio-cut polytope} ${\rm RCut}^K_n$ defined as the convex hull of ratio-cut vectors corresponding to
all assignments of $n$ points to at most $K$ nonempty clusters.
We are interested in a certain facet of ${\rm RCut}^K_n$, denoted by ${\rm RCut}^{=K}_n$, which corresponds to all ratio-cut vectors
of exactly $K$ nonempty clusters. As we detail later, ${\rm RCut}^{=K}_n$ is the polytope obtained by projecting out variables $Z_{ii}$, $i \in [n]$ from convex hull of the feasible region of Problem~\eqref{Kmeans2}.
We then study the facial structures of ${\rm RCut}^K_n$ and ${\rm RCut}^{=K}_n$ and
derive several classes of facet-defining inequalities for these polytopes. This in turn enables us to obtain a new LP relaxation
for K-means clustering. We then address the question of recovery when there are two clusters. First, we obtain a deterministic
sufficient condition under which the planted clusters correspond to an optimal solution of the LP relaxation. Subsequently,
we focus on the SBM, and prove that if $\Delta > 1+ \sqrt{3}$, the LP relaxation recovers the planted clusters with high probability. While this recovery guarantee is significantly better than the recovery guarantee of Problem~\eqref{badLP}, our empirical observations suggest that it is overly conservative. Indeed, our numerical experiments on a collection of randomly generated test problems indicate that the LP relaxation outperforms the Peng-Wei SDP relaxation.

\paragraph{\textbf{Organization.}}
The remainder of the paper is structured as follows. In Section~\ref{sec:ratioCutPoly}, we study the facial structure of the ratio-cut polytope.
In Section~\ref{sec:relax} we introduce a new LP relaxation for K-means clustering. We then focus on the case of two clusters and obtain a deterministic sufficient condition under which the planted clusters correspond to an optimal solution of the LP relaxation. In Section~\ref{sec:sbm}, we consider the K-means clustering problem with two clusters under the SBM. We first show that for dimension $m \geq 3$, in the continuum limit, the K-means clustering problem achieves optimal recovery threshold. Next, utilizing our deterministic condition of Section~\ref{sec:relax}, we obtain a recovery guarantee for the LP relaxation. We present our numerical experiments
in Section~\ref{sec:numerics}. Finally, proofs of the technical results omitted in Section~\ref{sec:sbm} are given in Section~\ref{sec:proofs}.

\section{The ratio-cut polytope}
\label{sec:ratioCutPoly}

In this section, we perform a polyhedral study of the convex hull of the feasible region of Problem~\eqref{Kmeans2}.
Denote by $\{x^i\}_{i=1}^n$ a set of points in $\R^m$ that we would like to put into at most $K$ clusters.
Consider a partition of $[n]$ denoted by $\{\Gamma_k\}_{k =1}^K$ where some of the partitions $\Gamma_k$ could be empty.
For any $1 \leq i < j \leq n$, define $X_{ij} = \frac{1}{|\Gamma_k|}$ if $i$ and $j$ belong to the same partition $\Gamma_k$, for some $k \in [K]$
and $X_{ij} = 0$ if $i$ and $j$ are in different partitions. Let $X$ be the $\binom{n}{2}$-vector whose elements are $X_{ij}$.
We refer to any such vector $X$ as a~\emph{ratio-cut vector}
and we refer to the convex hull in $\R^{\binom{n}{2}}$ of all ratio-cut vectors corresponding to \emph{at most} $K$ nonempty clusters as the~\emph{ratio-cut polytope}
and denote it by ${\rm RCut}^K_n$. Throughout the paper, we let $X_{ji} = X_{ij}$ whenever $i < j$.
If $K = 1$, then  ${\rm RCut}^K_n$ is given by
${\rm RCut}^1_n =\{X: X_{ij} = \frac{1}{n}, \; \forall 1 \leq i < j \leq n\}$. Moreover, it can be shown that  ${\rm RCut}^2_2 = \{X_{12}: 0 \leq X_{12} \leq 1/2\}$. Henceforth, we assume that
$n \geq 3$ and $2 \leq K \leq n$.

We denote by ${\rm RCut}^{=K}_n$ the convex hull of all ratio-cut vectors corresponding to \emph{exactly} $K$ nonempty clusters.
If $K = n$, it follows that ${\rm RCut}^{=n}_n = \{X: X_{ij} = 0, \; \forall 1 \leq i < j \leq n\}$
and if $K = n-1$, it can be shown that
$${\rm RCut}^{=n-1}_n = \Big\{X: \sum_{1 \leq i < j \leq n }{X_{ij}} =\frac{1}{2}, \; X_{ij} \geq 0, \; \forall 1 \leq i < j \leq n \Big\}.$$
Henceforth, when studying the facial structure of ${\rm RCut}^{=K}_n$, we assume that
$n \geq 4$ and $2 \leq K \leq n-2$. See Figure~\ref{figure5} for an illustration of ${\rm RCut}^{K}_n$ and ${\rm RCut}^{=K}_n$, with $n =3$ and $K= 2$.

\begin{figure}[htbp]
 \centering
 \epsfig{figure=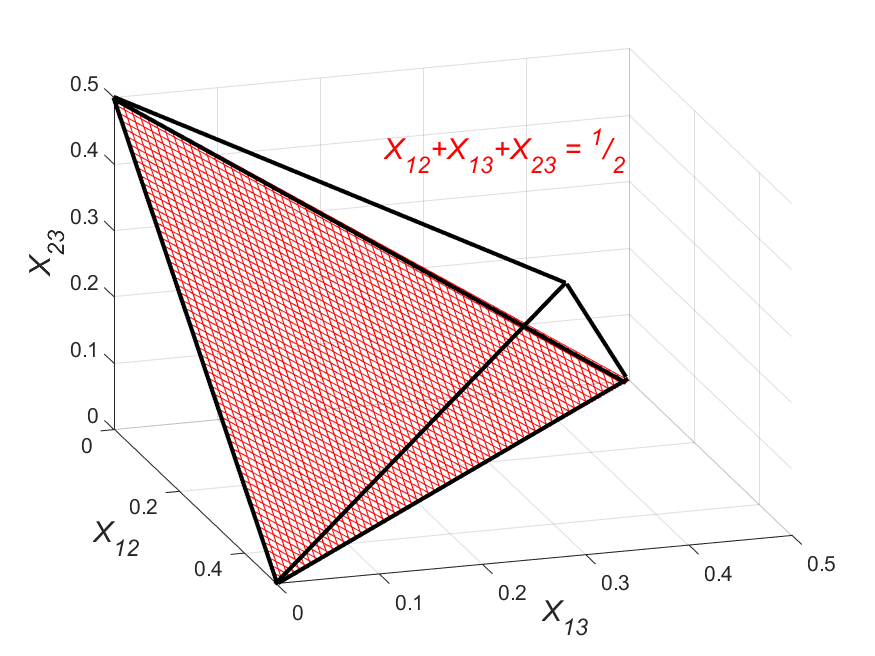, scale=0.5, trim=0mm 0mm 0mm 0mm, clip}
 \caption{The polytopes ${\rm RCut}^{K}_n$ and ${\rm RCut}^{=K}_n$, with $n =3$ and $K=2$. The ratio-cut polytope ${\rm RCut}^{K}_n$ is the convex hull of the ratio-cut vectors $\left\{(\frac{1}{3},\frac{1}{3},\frac{1}{3}), (\frac{1}{2},0,0), (0,\frac{1}{2},0), (0,0,\frac{1}{2})\right\}$. The polytope ${\rm RCut}^{=K}_n$, shaded in red, is a facet of ${\rm RCut}^{K}_n$ whose affine hall is given by
 $X_{12} + X_{13} + X_{23} = \frac{1}{2}$.}
\label{figure5}
\end{figure}

Consider the convex hull of the feasible region of Problem~\eqref{Kmeans2} denoted by $\Z^K_n$;
\ie the convex hull of all partition matrices $Z$.  It then follows that
\begin{equation}\label{projection}
{\rm RCut}^{=K}_n=\Big\{X \in \R^{\binom{n}{2}} | \; \exists Z \in \Z^K_n : X_{ij}=Z_{ij}, \forall 1 \leq i < j \leq n\Big\}.
\end{equation}
%
%
That is, the polytope ${\rm RCut}^{=K}_n$  is obtained by projecting out diagonal variables $Z_{ii}$, $i \in [n]$ from the convex hull of partition matrices.
Hence, understanding the facial structure of ${\rm RCut}^{=K}_n$ enables us to construct a strong LP relaxation for K-means clustering.
By definition, ${\rm RCut}^K_n \supset {\rm RCut}^{=K}_n$. However, we are interested in studying ${\rm RCut}^K_n$ due to its following
fundamental property.

\begin{proposition}\label{dimension}
The ratio-cut polytope ${\rm RCut}^K_n$ with $n \geq 3$ and $2 \leq K \leq n$ is full-dimensional; \ie $\dim({\rm RCut}^K_n) = \binom{n}{2}$.
\end{proposition}

\begin{proof}
Since ${\rm RCut}^K_n \subset {\rm RCut}^{K'}_n$ for all $K' > K$, to prove the statement, it suffices to show that ${\rm RCut}^2_n$
contains $\binom{n}{2} + 1$ affinely independent points.

First let $n \neq 4$.
Consider the ratio-cut vector corresponding to one cluster; \ie $X_{ij} = \frac{1}{n}$ for all $1 \leq i < j \leq n$
and the $\binom{n}{2}$ ratio-cut vectors corresponding to two clusters of cardinality two and $n-2$; \ie for any $1 \leq r < s \leq n$,
consider the ratio-cut vector with $X_{r s} = \frac{1}{2}$,  $X_{ij} = \frac{1}{n-2}$ if $i \notin \{r, s\}$ and $j \notin \{r, s\}$,
and $X_{ij} = 0$, otherwise. It can be checked that these $\binom{n}{2} + 1$ ratio-cut vectors are affinely independent.

Now consider $n = 4$; in this case ${\rm RCut}^2_4$ is the convex hull of eight ratio-cut vectors and it can be checked that any seven of these vectors containing the ratio-cut vector corresponding to one cluster are affinely independent. Notice that we could not utilize the same construction as the one we used for $n \neq 4$ since due to symmetry, for $n = 4$, the set of all  ratio-cut vectors corresponding to two clusters of cardinality two and $n-2$ consists of $\frac{1}{2} \binom{n}{2} = 3$ affinely independent points.
\end{proof}

\subsection{Three types of facets of ${\rm RCut}^K_n$ and ${\rm RCut}^{=K}_n$}
In the next three propositions, we present various classes of facet-defining inequalities for ${\rm RCut}^K_n$ and ${\rm RCut}^{=K}_n$.
These results enable us to construct a strong LP relaxation for K-means clustering.

\begin{proposition}\label{firstfacet}
The inequality
\begin{equation}\label{trivialFacet}
\sum_{1\leq i < j \leq n} {X_{ij}} \geq \frac{n-K}{2}
\end{equation}
is valid for ${\rm RCut}^K_n$ and is facet-defining if and only if $K \leq n-1$.
Moreover, the affine hull of ${\rm RCut}^{=K}_n$ is defined by the equation
\begin{equation}\label{affineHull}
\sum_{1\leq i < j \leq n} {X_{ij}} = \frac{n-K}{2}.
\end{equation}
\end{proposition}

\begin{proof}
Consider a partition of $[n]$ given by $\{\Gamma_k\}_{k=1}^K$; denote by $\K_1$ the subset of $[K]$ for which $|\Gamma_k| = 1$
and denote by $\K_2$ the subset of $[K]$ for which $|\Gamma_k| \geq 2$.
It then follows that
$$\sum_{1\leq i < j \leq n} {X_{ij}} = \sum_{k\in \K_2} {\frac{1}{|\Gamma_k|} \binom{|\Gamma_k|}{2}} = \frac{1}{2}\Big(\sum_{k\in \K_2} {|\Gamma_k|}-|\K_2|\Big) = \frac{n-|\K_1|-|\K_2|}{2} \geq \frac{n-K}{2},$$
where the last inequality holds with equality when $K = |\K_1| + |\K_2|$; \ie when all $\Gamma_k$, $k \in [K]$ are nonempty; that is,
we have exactly $K$ clusters.

Now let $K \leq n-1$ and consider a facet-defining inequality $a X \geq \alpha$
for ${\rm RCut}^K_n$ that is satisfied tightly by all ratio-cut vectors that are binding for inequality~\eqref{trivialFacet}.
We show that the two inequalities coincide up to a positive scaling which by full dimensionality of the ratio-cut polytope (see Proposition~\ref{dimension}), implies that inequality~\eqref{trivialFacet} defines a facet of ${\rm RCut}^K_n$.

Consider a partition of $[n]$ with $K$ nonempty clusters $\Gamma_k$, $k \in [K]$ with $\Gamma_1 = \{i\}$ for some $i \in [n]$.
Consider a second partition obtained by switching point $i$
with a point $j$ where $j \in \Gamma_r$ for some $|\Gamma_r| \geq 2$. Note
such a $\Gamma_r$ always exists since by assumption $K \leq n-1$.
Substituting the corresponding ratio-cut vectors in $a X = \alpha$ and subtracting the resulting equalities, we obtain:
\begin{equation}\label{aux1}
\sum_{l \in \Gamma_r \setminus \{i,j\}} {a_{jl}} = \sum_{l \in \Gamma_r \setminus \{i,j\}} {a_{il}}.
\end{equation}
Next, take the two partitions defined above and for each one, remove a point $k \neq i,j$ from  $\Gamma_r$ and place it in $\Gamma_1$.
Substituting the corresponding ratio-cut vectors in $a X = \alpha$  yields:
\begin{equation}\label{aux2}
\frac{a_{ik}}{2} + \frac{\sum_{l \in \Gamma_r \setminus \{i,j,k\}} {a_{jl}}}{|\Gamma_r|-1} = \frac{a_{jk}}{2} + \frac{\sum_{l \in \Gamma_r \setminus \{i,j, k\}} {a_{il}}}{|\Gamma_r|-1}.
\end{equation}
From~\eqref{aux1} and~\eqref{aux2} it follows that $a_{ik} = a_{jk}$ for all distinct $i,j, k \in [n]$. Moreover it can be checked that for any ratio-cut vector
associated to $K$ nonempty clusters we have $\sum_{1 \leq i < j \leq n} {X_{ij}} = (n-K)/2$. Hence, if $K \leq n-1$, the inequality $a X \geq \alpha$ coincides with~\eqref{trivialFacet}
up to a positive scaling implying that inequality~\eqref{trivialFacet} is facet-defining.

Now let $n = K$; in this case, the right-hand side of inequality~\eqref{trivialFacet}
is zero and hence is implied by valid inequalities $X_{ij} \geq 0$ for all $1 \leq i < j \leq n$. Therefore, in this case
inequality~\eqref{trivialFacet} is not facet-defining.

Finally, since the set of all ratio-cut vectors corresponding to $K$ nonempty clusters
constitute the set of tight points of the facet-defining inequality~\eqref{trivialFacet}, we conclude that $\dim({\rm RCut}^{=K}_n) = \binom{n}{2} -1 $ for all $K \leq n-1$ and its affine hull is
induced by $\sum_{1\leq i < j \leq n} {X_{ij}} = (n-K)/2$.
\end{proof}

Next, we present a class of facet-defining inequalities for the ratio-cut polytope that can be considered as the generalization
of the well-known triangle inequalities associated with the cut polytope~\cite{dl:97} (see Subsection~\ref{compareCut} for more detail).

\begin{proposition}\label{newf}
Let $l \in [n]$ and let $T$ be a nonempty subset of $[n] \setminus \{l\}$.  Then the inequality
\begin{equation}\label{facets}
2 \sum_{j \in T} {X_{lj}} + \sum_{j \in [n] \setminus T \cup \{l\}} {X_{lj}} \leq 1 + \sum_{i,j \in T: i < j} {X_{ij}},
\end{equation}
is valid for ${\rm RCut}^K_n$. Moreover, if $2 \leq |T| \leq K$, then inequality~\eqref{facets}  defines a facet of ${\rm RCut}^K_n$,
and if in addition, $K \leq n-2$ then this inequality defines a facet of ${\rm RCut}^{=K}_n$.
\end{proposition}

\begin{proof}
Denote by $\omega$ the number of points in $T$ that are in the same cluster as $l$ and let $\Omega$ denote the size of the corresponding cluster; clearly $\omega \geq 0$ and
$\Omega \geq 1$.
Then $\sum_{j \in T} {X_{lj}} = \frac{\omega}{\Omega}$ and $\sum_{j \in [n] \setminus T \cup \{l\}} {X_{lj}} = \frac{\Omega-\omega-1}{\Omega}$.
Moreover,  if $\omega \geq 2$, we have $\sum_{i,j \in T} {X_{ij}} \geq \binom{\omega}{2} \frac{1}{\Omega}$. Hence to show the
validity of inequality~\eqref{facets} it suffices to show that $1+ \frac{\omega-1}{\Omega} \leq 1$, if $\omega \in \{0,1\}$ and
 $1+ \frac{\omega-1}{\Omega} \leq 1+\frac{\omega(\omega-1)}{2\Omega}$, if $\omega \geq 2$, both of which are valid.

Let $K \in \{2, \ldots, n\}$; we now show that if $t:= |T| \in \{2, \ldots, \min\{n-1,K\}\}$,
then inequality~\eqref{facets} defines a facet of ${\rm RCut}^K_n$.
Denote by
\begin{equation}\label{samef}
a X \leq \alpha,
\end{equation}
a nontrivial valid inequality for ${\rm RCut}^K_n$ that is satisfied tightly by all ratio-cut vectors
that are binding for~\eqref{facets}. We show that inequalities~\eqref{facets} and~\eqref{samef} coincide up
to a positive scaling, which by full dimensionality of ${\rm RCut}^K_n$ (see Proposition~\ref{dimension})
implies that~\eqref{facets} defines a facet of ${\rm RCut}^K_n$.

Let $\Gamma_k$, $k \in [K]$ form a partition of $[n]$.
Assume that (i) $l \in \Gamma_1$, (ii) $|\Gamma_1 \cap T| = 1$ or $|\Gamma_1 \cap T| = 2$
and (iii) $|\Gamma_k \cap T| \leq 1$  for all $k \in [K] \setminus \{1\}$.
It can be checked that the ratio-cut vector corresponding to any partition satisfying conditions~(i)-(iii)
satisfies~\eqref{facets} tightly.
Note that all such partitions consist of at least $t$ nonempty clusters if $|\Gamma_1 \cap T| =1$ and at least $t-1$ nonempty clusters if
$|\Gamma_1 \cap T| = 2$, all of which correspond to valid ratio-cut vectors in  ${\rm RCut}^K_n$ since by assumption $2 \leq t \leq K$.
Henceforth, by a ``binding partition,'' we imply a partition of $[n]$ consisting of at most $K$ nonempty clusters, which
satisfies conditions~(i)-(iii) above.
In the following we present several types of binding partitions and by substituting the corresponding ratio-cut vectors
in $a X = \alpha$, we prove the facetness of inequality~\eqref{facets}.

Take some $r \in T$ and some $s \in [n] \setminus (T \cup \{l\})$. Consider a binding partition consisting of $q \leq K-1$ nonempty clusters
such that $\Gamma_2 = \{r,s\}$. Note that this binding partition exists since by assumption $t \geq 2$.
Moreover, such a partition with exactly $K-1$ nonempty
clusters exists if $K \leq n-1$. Now consider another binding partition consisting of  $q+1$ nonempty clusters,
obtained from the above partition by removing $s$ from $\Gamma_2$ and putting it in $\Gamma_{q+1}$. Substituting the ratio-cut vectors
in $a X = \alpha$ we obtain $a_{rs} = 0$.

For any $j \in [n] \setminus (T \cup \{l,s\})$, consider a binding partition such that $\Gamma_2 = \{r,s\}$ and
$\Gamma_3 = \{j\}$. Such a partition
with exactly $K$ nonempty clusters exists if $K \leq n-2$. Construct a second binding partition from the above partition
by swapping $j$ and $s$. Substituting the ratio-cut vectors
in $a X = \alpha$ and using $a_{rs} = 0$, we obtain $a_{rj} = 0$ for all $j \in [n] \setminus (T \cup \{l\})$.
Similarly, we obtain $a_{sj} = 0$ for all $j \in [n] \setminus (T \cup \{l,s\})$.

For any  $i \in T$  and for any $j \in [n] \setminus (T \cup \{l\})$, consider a binding partition
(consisting of $K$ nonempty clusters, if $K \leq n-2$)
such that $\Gamma_2 = \{r,j\}$ and $\Gamma_3 = \{i\}$.
Construct a second binding partition from the above partition
by swapping $r$ and $i$. Substituting the ratio-cut vectors
in $a X = \alpha$ and using $a_{rj} = 0$, we obtain
\begin{equation}\label{f1}
a_{ij} = 0, \quad \forall i \in T, \; j \in [n] \setminus (T \cup \{l\}).
\end{equation}

For any  $i, j \in [n] \setminus (T \cup \{l\})$, consider a binding partition
(consisting of $K$ nonempty clusters, if $K \leq n-2$)
such that $\Gamma_2 = \{s,j\}$ and $\Gamma_3 = \{i\}$.
Construct a second binding partition from the above partition
by swapping $s$ and $i$. Substituting the ratio-cut vectors
in $a X = \alpha$ and using $a_{sj} = 0$, we obtain
\begin{equation}\label{f2}
a_{ij} = 0, \quad \forall i, j \in [n] \setminus (T \cup \{l\}).
\end{equation}

For any $i, j \in T$, consider a binding partition
(consisting of $K$ nonempty clusters, if $K \leq n-1$)
with $\Gamma_1 = \{1,i\}$ and $\Gamma_2 = \{j\}$.
Construct a second binding partition from the above partition
by swapping $i$ and $j$. Substituting the ratio-cut vectors
in $a X = \alpha$, we obtain
\begin{equation}\label{f3}
a_{li} = \beta, \quad \forall i \in T.
\end{equation}

For any $i, j \in [n] \setminus (T \cup \{l\})$, consider a binding partition
(consisting of $K$ nonempty clusters, if $K \leq n-1$)
with $|\Gamma_1 \cap T| = 1$, $i \in \Gamma_1$, and $j \in \Gamma_2$.
Construct a second binding partition from the above partition
by swapping $i$ and $j$. Substituting the ratio-cut vectors
in $a X = \alpha$ and using~\eqref{f1}-\eqref{f3}, we obtain
\begin{equation}\label{f4}
a_{li} = \gamma, \quad \forall i \in [n] \setminus (T \cup \{l\}).
\end{equation}

For any $i, j, k \in T$, consider a binding partition
(consisting of $K$ nonempty clusters, if $K \leq n-2$)
with $\Gamma_1 = \{l,i,j\}$ and $\Gamma_2 = \{k\}$.
Construct a second binding partition from the above partition
by swapping $j$ and $k$. Substituting the ratio-cut vectors
in $a X = \alpha$ and using~\eqref{f3}, we obtain
\begin{equation}\label{f5}
a_{ij} = \zeta, \quad \forall i, j \in T.
\end{equation}

For any $i, j \in T$ and $k \in [n] \setminus (T \cup \{l\})$, consider a binding partition
(consisting of $K$ nonempty clusters, if $K \leq n-2$)
with $\Gamma_1 = \{l,i,j\}$ and $\Gamma_2 = \{k\}$.
Construct a second binding partition from the above partition
by swapping $j$ and $k$. Substituting the ratio-cut vectors
in $a X = \alpha$ and using~\eqref{f1}-~\eqref{f5}, we obtain
\begin{equation}\label{f6}
\beta +\zeta = \gamma .
\end{equation}

For any $i, j \in T$ and $k \in [n] \setminus (T \cup \{l\})$, consider a binding partition
(consisting of $K$ nonempty clusters, if $K \leq n-2$)
with $\Gamma_1 = \{l,i\}$ and $\Gamma_2 = \{j,k\}$.
Construct a second binding partition from the above partition
by adding $j$ to $\Gamma_1$. Substituting the ratio-cut vectors
in $a X = \alpha$ and using~\eqref{f1},~\eqref{f3}, and~\eqref{f5}, we obtain
\begin{equation}\label{f7}
\beta + 2 \zeta = 0.
\end{equation}

Consider a binding partition, consisting of $K$ nonempty clusters if $K \leq n-1$,
with $\Gamma_1 = \{l,i\}$ for some $i \in T$. Substituting
the corresponding ratio cut vector in $a X = \alpha$ and using~\eqref{f1}-\eqref{f3} yields $\beta = 2 \alpha$.
Together with~\eqref{f6} and~\eqref{f7}, this in turn implies that
inequality~\eqref{samef} can be equivalently written as
$\alpha (2 \sum_{j \in T} {X_{lj}} + \sum_{j \notin T \cup\{l\}} {X_{1j}} - \sum_{i,j \in T, i < j}{X_{ij}}) \leq \alpha$
where $\alpha > 0$ since by assumption, inequality~\eqref{samef} is nontrivial and inequality~\eqref{facets} is valid, and this concludes the proof of facetness for ${\rm RCut}^K_n$.

Finally, let us consider the polytope ${\rm RCut}^{=K}_n$ for $2 \leq K \leq n-2$.
In the above proof, with the exception of the first one, all of the binding partitions
consist of $K$ nonempty clusters and hence the corresponding ratio-cut vectors are present in ${\rm RCut}^{=K}_n$.
As ${\rm RCut}^{=K}_n$ is a facet of  ${\rm RCut}^K_n$, it follows that the inequality~\eqref{facets} defines a facet of ${\rm RCut}^{=K}_n$..
\end{proof}

Now consider the valid inequalities $X_{ij} \geq 0$ for all $1 \leq i < j \leq n$.
As we show in the following,
if $K \geq 3$, these inequalities define facets of the ratio-cut polytope.

\begin{proposition}\label{triv}
Let $K \leq n-1$. Then inequalities $X_{ij} \geq 0$ for all $1 \leq i < j \leq n$ define
facets of ${\rm RCut}^K_n$ and ${\rm RCut}^{=K}_n$ if and only if $K \geq 3$.
\end{proposition}

\begin{proof}
We show without loss of generality that $X_{12} \geq 0$ defines a facet of ${\rm RCut}^K_n$ if and only if  $K \geq 3$.
First suppose that $K \geq 3$.
It then follows that $X_{12} \geq 0$ is binding at all ratio-cut vectors in which $1$ and $2$ are not in the same cluster. Let $a X \geq \alpha$
denote a nontrivial valid inequality for ${\rm RCut}^K_n$ that is binding at all ratio-cut vectors for which $X_{12} \geq 0$ is satisfied tightly.
We show that the two inequalities coincide up
to a positive scaling, which by full dimensionality of ${\rm RCut}^K_n$ implies that $X_{12} \geq 0$ defines a facet of ${\rm RCut}^K_n$.

Consider a partition of $[n]$ consisting of $K-1$ nonempty clusters such that $\Gamma_1 = \{1, n\}$.
Construct a second partition of $[n]$ consisting of $K$ nonempty clusters obtained from the partition defined above by removing the $n$th point from the first cluster and putting it in the $K$the cluster.
The corresponding ratio-cut vectors are binding and hence substituting in $a X = \alpha$
we obtain $a_{1n} = 0$.

Next for any $j \in [n] \setminus \{1,2,n\}$ consider two partitions of $[n]$ consisting of $K$ nonempty clusters corresponding to binding ratio-cut vectors, where in the first partition we have $\Gamma_1 = \{1, n\}$, $\Gamma_2 = \{j\}$ and $\cup_{k=3}^{K}{\Gamma_k} =  [n] \setminus \{1,j,n\}$.
The second partition is obtained from the first one by only modifying the first and second clusters as follows: $\Gamma_1 = \{1,j\}$,
$\Gamma_2 = \{n\}$. Substituting the ratio-cut vectors in $a X = \alpha$ and using $a_{1n} = 0$ yields:
\begin{equation}\label{triv3}
a_{1j} = 0, \quad \forall j \in [n]\setminus \{1,2\}.
\end{equation}
Similarly, for any $j \in [n] \setminus \{1,2\}$ and for any $k \in [n] \setminus \{1,j\}$,
consider two partitions of $[n]$ consisting of $K$ nonempty clusters both of which correspond to binding ratio-cut vectors, defined as follows: in the first partition we have  $\Gamma_1 = \{1, j\}$, $\Gamma_2 = \{k\}$, and $\cup_{k=3}^{K}{\Gamma_k} =  [n] \setminus \{1,j,k\}$.
The second partition is obtained from the first one by only changing the first and second clusters as follows: $\Gamma_1 = \{1\}$,
$\Gamma_2 = \{j, k\}$.
Substituting the ratio-cut vectors in $a X = \alpha$ and using~\eqref{triv3} yields
$$
a_{jk} = 0, \quad \forall 2 \leq i < j \leq n.
$$
Hence, since by assumption $a X \geq \alpha$ is nontrivial and valid for ${\rm RCut}^K_n$
we conclude that it can be equivalently written as $a_{12} X_{12} \geq 0$
where $a_{12} > 0$, implying that if $K \geq 3$ the inequality $X_{12} \geq 0$ defines a facet of ${\rm RCut}^K_n$.

Now suppose that $K = 2$. We show that $X_{12} \geq 0$ is implied by a collection of inequalities all of which are valid for ${\rm RCut}^2_n$ indicating that it is not facet-defining. Consider the following inequalities
\begin{equation}\label{aux}
2(X_{1i} + X_{2i}) + \sum_{j\in [n] \setminus \{1,2,i\}}{X_{ij}} \leq 1 + X_{12},  \quad \forall i \in \{3,\ldots, n\}.
\end{equation}
By letting $S = \{1,2,i\}$ for each $i \in \{3, \ldots, n\}$, from Proposition~\ref{newf} it follows that
inequalities~\eqref{aux} are valid for ${\rm RCut}^2_n$.
Moreover, take inequality $\sum_{1 \leq i < j \leq n} {X_{ij}} \geq (n-2)/2$ whose validity follows from Proposition~\ref{firstfacet}; multiplying this inequality by $-2$ and adding the result to the inequality obtained by summing up all inequalities~\eqref{aux}, we get $-n X_{12} \leq 0$. Hence $X_{ij} \geq 0$ does not define a facet of ${\rm RCut}^2_n$.

Finally, let us consider the polytope ${\rm RCut}^{=K}_n$ for $3 \leq K \leq n-1$.
In the above proof, with the exception of the first one, all of the defined ratio-cut vectors binding for $X_{ij} \geq 0$
correspond to $K$ nonempty clusters and hence are present in ${\rm RCut}^{=K}_n$.
As ${\rm RCut}^{=K}_n$ is a facet of  ${\rm RCut}^K_n$, it follows that the inequality $X_{ij} \geq 0$ defines a facet of ${\rm RCut}^{=K}_n$ as well.
\end{proof}

\subsection{Connections with the cut polytope}
\label{compareCut}

The cut polytope, a celebrated polytope in combinatorial optimization, is the convex hull of all cut vectors of a graph
(see~\cite{dl:97} for an exposition). This polytope corresponds to the convex hull of the feasible region of the max-cut problem.
Consider a graph with $n$ nodes; we define a~\emph{cut vector} $Y$ associated with a graph cut as follows: for any $1 \leq i < j \leq n$, let
$Y_{ij} = 1$ if $i$ and $j$ are in the same partition and let $Y_{ij} = 0$, otherwise. The cut polytope ${\rm Cut}_n$ is then defined
as the convex hull of all cut vectors. In this section, we describe some key similarities and differences between the cut polytope and the ratio-cut polytope.

The facial structure of the cut polytope has been extensively studied and numerous classes of facet-defining inequalities have been identified
for this polytope. The most well-known class of facet-defining inequalities for ${\rm Cut}_n$ are \emph{triangle inequalities} given by
\begin{equation}\label{tri2}
Y_{ij} + Y_{ik} \leq 1 + Y_{jk},
\end{equation}
and
\begin{equation}\label{tri1}
Y_{ij} + Y_{ik} + Y_{jk} \geq 1,
\end{equation}
for all distinct $i, j, k \in [n]$ (see for example~\cite{dl:97}).

Now let $l \in [n]$, and let $T \subseteq [n] \setminus \{l\}$ with $ 2\leq |T| \leq K$; define $X_{ll} := 1 - \sum_{j \in [n] \setminus \{l\}} {X_{lj}}$. Then inequality~\eqref{facets} can be equivalently written as
\begin{equation}\label{gtri}
\sum_{j \in T} {X_{lj}} \leq X_{ll} + \sum_{i,j \in T : i < j}{X_{ij}}.
\end{equation}
By Proposition~\ref{newf}, inequalities of the form~\eqref{gtri} define facets of ${\rm RCut}^{K}_n$. Let $K = 2$; then all inequalities
of the form~\eqref{gtri} can be written as:
\begin{equation}\label{gtriK2}
X_{ij} + X_{ik} \leq X_{ii} + X_{jk}, \quad \forall \; {\rm distinct} \quad i,j,k \in [n],
\end{equation}
where as before $X_{ii} := 1 - \sum_{j \in [n] \setminus \{i\}} {X_{ij}}$.
Hence, inequalities~\eqref{gtriK2} (and more generally, inequalities~\eqref{gtri}) for the ratio-cut polytope can be considered as an equivalent of inequalities~\eqref{tri2} for the cut polytope.
Moreover, the fact-defining inequality~\eqref{trivialFacet} can be considered as an equivalent of
inequalities~\eqref{tri1} for the cut polytope. We further detail on some parallels between ${\rm Cut}_n$ and  ${\rm RCut}^{K}_n$ in the next section as we introduce a relaxation of the ratio-cut polytope.

In spite of certain similarities, it is important to note that there are fundamental differences between the cut polytope and the ratio-cut polytope. For example, unlike the cut polytope, the ratio-cut polytope does not have the so-called ``zero-lifting'' property (see Section~26.5 of~\cite{dl:97} for the definition of zero-lifting). Roughly speaking, the zero-lifting property implies that if an inequality defines a facet of
${\rm Cut}_n$ then it also defines a facet of ${\rm Cut}_{n'}$ for all $n' > n$.
Indeed, the lack of this property in case of inequalities~\eqref{gtri} is hidden in the definition of $X_{ii}$.

\subsection{A polyhedral relaxation of the ratio-cut polytope}
We now define a polyhedral relaxation of ${\rm RCut}^K_n$ defined by all facet-defining inequalities given by Propositions~\ref{firstfacet}, \ref{newf}, and~\ref{triv}.
We denote this relaxation by ${\rm RMet}^K_n$ due to its similarities to the \emph{metric polytope}.
The metric polytope defined by triangle inequalities~\eqref{tri2} and~\eqref{tri1} is a widely used relaxation of the cut polytope.
Similarly, we define a polyhedral relaxation of ${\rm RCut}^{=K}_n$ defined by equality~\eqref{affineHull}
together with facet-defining inequalities given by Propositions~\ref{newf}, and~\ref{triv}.
We denote this relaxation by ${\rm RMet}^{=K}_n$. Notice that ${\rm RMet}^K_n$ is defined by a collection of inequalities all of which are facet-defining for ${\rm RCut}^K_n$ and ${\rm RMet}^{=K}_n$ is the restriction of ${\rm RMet}^K_n$ to one of such inequalities. It then follows that ${\rm RMet}^{=K}_n$ is a facet of ${\rm RMet}^K_n$.

It is well-known that the metric polytope coincides with the cut polytope if and only if $n \leq 4$~\cite{dl:97}. The following demonstrates an analogous relation between
${\rm RCut}^K_n$ and ${\rm RMet}^K_n$.
\begin{proposition}\label{tightness}
${\rm RCut}^K_n = {\rm RMet}^K_n$ if and only if $n \leq 4$.
\end{proposition}
\begin{proof}
If $n = 3$ and $K = \{2,3\}$ or if $n = 4$ and $K \in \{2,3,4\}$, it can be checked, by direct calculation that the ratio-cut vectors constitute all vertices of ${\rm RMet}^K_n$, implying ${\rm RCut}^K_n = {\rm RMet}^K_n$.

Now let $n \geq 5$ and consider some $K \in \{2, \ldots, n\}$.
To show that ${\rm RCut}^K_n \subset {\rm RMet}^K_n$, we present a point $\bar X$ such that $\bar X \in {\rm RMet}^K_n$
and $\bar X \notin {\rm RCut}^K_n$. To prove the latter, we present an inequality that is valid for ${\rm RCut}^K_n$
but is violated by $\bar X$.

We first give the valid inequality for the ratio-cut polytope. Consider a pair $i,j \in [n]$
and let $T \subseteq [n] \setminus \{i,j\}$  with $|T| \geq 2$.
Let $X_{ii} = 1-\sum_{k\in [n] \setminus \{i\}}{X_{ik}}$
and suppose that $X_{jj}$ is similarly defined.
We first show that the inequality
\begin{equation}\label{valid}
\sum_{k \in T}{(X_{ik} +X_{jk})} - X_{ij} \leq X_{ii} + X_{jj} + \sum_{k <l \in T} {X_{kl}}
\end{equation}
is valid for ${\rm RCut}^K_n$. Two cases arise:
\begin{itemize}[leftmargin=*]
\item [(i)] $i$ and $j$ are in the same cluster of size $\Omega$.
Then $X_{ii} = X_{jj} = X_{ij} = \frac{1}{\Omega}$. Denote by $\omega$ the number of points in $T$ that belong to the same cluster
as $i$ and $j$. Then $\sum_{k \in T}{(X_{ik} +X_{jk})} = \frac{2\omega}{\Omega}$. Moreover, if $\omega \geq 2$, we have
$\sum_{k <l \in T} {X_{kl}} \geq \binom{\omega}{2}\frac{1}{\Omega}$. Hence if $\omega \leq 1$, it suffices to show that
$\frac{2\omega}{\Omega} - \frac{1}{\Omega} \leq \frac{1}{\Omega} + \frac{1}{\Omega}$ and if $\omega \geq 2$, it suffices to show that
$\frac{2\omega}{\Omega} - \frac{1}{\Omega} \leq \frac{1}{\Omega} + \frac{1}{\Omega} + \binom{\omega}{2}\frac{1}{\Omega}$, both of which are valid.

\item [(ii)] $i$ and $j$ are in two distinct clusters of size $\Omega_1$ and $\Omega_2$, respectively.
Then $X_{ii} = \frac{1}{\Omega_1}$, $X_{jj} = \frac{1}{\Omega_2}$ and $X_{ij} = 0$.
Denote by $\omega_1$ (resp. $\omega_2$) the number of points in $T$ that are in the same cluster with $i$ (resp. $j$).
Then $\sum_{k \in T}{(X_{ik} +X_{jk})} = \frac{\omega_1}{\Omega_1} +  \frac{\omega_2}{\Omega_2}$.
Let $q_1 = \binom{\omega_1}{2}$ (resp. $q_2 = \binom{\omega_2}{2}$),
if $\omega_1 \geq 2$ (resp. $\omega_2 \geq 2$) and let $q_1 = 0$ (resp. $q_2 = 0$) , otherwise.
Then it suffices to show that
$\frac{\omega_1}{\Omega_1} +  \frac{\omega_2}{\Omega_2} \leq \frac{1}{\Omega_1} + \frac{1}{\Omega_2} + \frac{q_1}{{\Omega_1}}
+ \frac{q_2}{{\Omega_2}}$.
The validity of this statement follows from the fact that $\omega_1 \leq 1 + q_1$ and $\omega_2\leq 1 + q_2$.

\end{itemize}

We now present a point $\bar X \in {\rm RMet}^K_n$ that does not satisfy inequality~\eqref{valid}.
Suppose that $|T| = 3$; for notational simplicity, let $\{i,j\} = \{1,2\}$ and $T = \{3,4,5\}$.
Two cases arise:
\begin{itemize}[leftmargin=*]
\item [(i)] $n =5$: in this case, let
\begin{eqnarray*}
  \bar X_{12} &=& 0 \\
  \bar X_{1k} =\bar X_{2k} &=& \alpha = \frac{5}{24}, \quad k \in \{3,4,5\}\\
  \bar X_{kl} &=& \omega = \frac{2}{24} , \quad k < l \in \{3,4,5\}
\end{eqnarray*}
We now show that $\bar X \in {\rm RMet}^K_5$.  We have
$\sum_{1 \leq i < j \leq 5} {\bar X_{ij}} = 6  \alpha + 3 \omega = \frac{3}{2} \geq \frac{5-K}{2}$,
where the last inequality is valid for any $K \geq 2$; hence, inequality~\eqref{trivialFacet} is satisfied at $\bar X$.
It remains to show the validity of inequalities~\eqref{gtri} for $ 2 \leq |T| \leq \min\{K, 4\}$
and for $2\leq K \leq 5$. We have $\bar X_{11} = \bar X_{22} = \beta = \frac{9}{24}$
and $\bar X_{kk} = \gamma = \frac{10}{24}$ for $k \in \{3,4,5\}$. Then for $|T| = 2$, it suffices to have
$2 \alpha \leq \beta + \omega$, $2 \alpha \leq \gamma$, $\omega \leq \gamma$,
all of which are valid. For $|T| = 3$, it suffices to have $3 \alpha \leq \beta + 3 \omega$ which is satisifed
and all inequalities corresponding to $|T| = 4$ are implied by the above inequalities. Thus, we conclude that $\bar X \in {\rm RMet}^K_5$.
Substituting $\bar X$ in inequality~\eqref{valid}
yields $6 \alpha - 0 \leq 2 \beta + 3 \omega$,
which simplifies to $\frac{30}{24} \leq \frac{18}{24}+ \frac{6}{24}$ and is clearly not valid.

\item [(ii)] $n > 5$: in this case, let
\begin{eqnarray*}
  \bar X_{12} &=& 0 \\
  \bar X_{1k} =\bar X_{2k} &=&  \alpha = \frac{n}{3(3n-7)}, \quad k \in \{3,4,5\}\\
  \bar X_{kl} &=& \omega = \frac{n}{6(3n-7)} , \quad k < l \in \{3,4,5\}\\
  \bar X_{1k} =\bar X_{2k} &=&  \eta_1 = \frac{3n-14}{2(3n-7)(n-5)}, \quad k \in \{6,\ldots,n\}\\
   \bar X_{kl}& = & \eta_2 = \frac{4n-21}{3(3n-7)(n-5)}, \quad   k \in \{3,4,5\}, \; l \in \{6,\ldots,n\} \\
   \bar X_{kl} & = & \eta_3 = \frac{3n-14}{(3n-7)(n-5)}, \quad k < l \in \{6,\ldots,n\}.
\end{eqnarray*}
We now show that $\bar X \in {\rm RMet}^K_n$. First, $\bar X_{ij} \geq 0$ for all $1 \leq i < j \leq n$.
Moreover,
$\sum_{1 \leq i < j \leq n} {\bar X_{ij}} = 6 \alpha + 3 \omega  + 2(n-5) \eta_1 + 3(n-5) \eta_2 +\frac{(n-5)(n-6)}{2} \eta_3
= \frac{2 n}{3n-7} + \frac{n}{2(3n-7)}  + \frac{3n-14}{3n-7} +\frac{4n-21}{3n-7} +\frac{(n-6)}{2}  \frac{3n-14}{3n-7} = \frac{n-2}{2} \geq \frac{n-K}{2}$, implying inequality~\eqref{trivialFacet} is satisfied.
We now establish the validity of inequalities~\eqref{gtri} at $\bar X$, for $ 2 \leq |T| \leq K$.
It can be checked that $\bar X_{11} = \bar X_{22} = \beta= \frac{n}{2(3n-7)}$ and
$\bar X_{kk} = \gamma= \frac{2 n}{3(3n-7)}$ for $k \in \{3, 4, 5\}$.
From the construction of $\bar X$, it follows that the inequalities of the form~\eqref{gtri} not implied by the rest are the following:
$2 \alpha \leq \beta + \omega$, $2 \alpha \leq \gamma$, $\omega \leq \gamma$,
$3 \alpha \leq \beta + 3 \omega$ all of which are satisfied by $\bar X$.
Hence, $\bar X \in {\rm RMet}^K_n$.
Finally, substituting $\bar X$ in inequality~\eqref{valid}
yields $6 \frac{n}{3(3n-7)} - 0 \leq 2\frac{n}{2(3n-7)}+ 3 \frac{n}{6(3n-7)}$,
which simplifies to $2 n \leq n+ \frac{n}{2}$ and is clearly not valid.
\end{itemize}

\end{proof}

While for $n \geq 5$, the polytopes ${\rm RMet}^K_n$ and  ${\rm RCut}^K_n$ do not coincide,
next, we show that every ratio-cut vector is a vertex of ${\rm RMet}^K_n$.

\begin{proposition}
Every ratio-cut vector is a vertex of ${\rm RMet}^K_n$ and every ratio-cut vector corresponding to $K$ nonempty clusters is
a vertex of ${\rm RMet}^{=K}_n$.
\end{proposition}

\begin{proof}
We first show that every ratio-cut vector is a vertex of ${\rm RMet}^K_n$.
By Proposition~\ref{tightness} it suffices to consider $n \geq 5$.
Let $\Gamma_k$, $k \in [K]$ form a partition of $[n]$, where as before we
allow for some empty $\Gamma_k$.
Denote by $K'$ the number of nonempty clusters and denote by $\hat X$ the corresponding ratio-cut vector.
We show that $\hat X$ is a vertex of ${\rm RMet}^K_n$ by presenting
$\binom{n}{2}$ linearly independent facets of  ${\rm RMet}^K_n$ that are satisfied tightly by $\hat X$.
The following cases arise:
\begin{itemize}[leftmargin=*]
\item [(i)] $K' = 1$: in this case for each $i,j,k$ in the partition,
consider facet-defining inequalities of the form~\eqref{facets} given by
$2(X_{ij} + X_{ik}) + \sum_{l \in [n] \setminus \{j,k\}}{X_{il}} \leq 1 + X_{jk}$,
$2(X_{ij} + X_{jk}) + \sum_{l \in [n] \setminus \{i,k\}}{X_{jl}}\leq 1 + X_{ik}$, and
$2(X_{ik} + X_{jk})  + \sum_{l \in [n] \setminus \{i,j\}}{X_{kl}}\leq 1 + X_{ij}$.

\item [(ii)] $K = K' = 2$: in this case for each $i,j$ in one partition and for each $k$ in the other partition
consider facet-defining inequalities of the form~\eqref{facets} given by
$2(X_{ij} + X_{ik}) + \sum_{l \in [n] \setminus \{j,k\}}{X_{il}} \leq 1 + X_{jk}$
and $2(X_{ij} + X_{jk}) + \sum_{l \in [n] \setminus \{i,k\}}{X_{jl}}\leq 1 + X_{ik}$.
In addition, consider the facet defining inequality $\sum_{1 \leq i < j \leq n}{X_{ij}} \geq \frac{n}{2}-1$.

\item [(iii)] $K > 2$, $K' \neq 1$, in this case for each $i,j$ in the same partition and for any $k$ in a different partition
consider facet-defining inequalities of the form~\eqref{facets} given by
$2(X_{ij} + X_{ik}) + \sum_{l \in [n] \setminus \{j,k\}}{X_{il}} \leq 1 + X_{jk}$
and $2(X_{ij} + X_{jk}) + \sum_{l \in [n] \setminus \{i,k\}}{X_{jl}}\leq 1 + X_{ik}$.
In addition, for each $i,j$ not in the same partition consider the facet defining inequality $X_{ij} \geq 0$.
\end{itemize}
It can be checked that $\hat X$ satisfies above
inequalities tightly and that these inequalities contain $\binom{n}{2}$ linearly independent facets implying that $\hat X$
is a vertex of ${\rm RMet}^K_n$.
%
%
%

Finally, notice that every ratio-cut vector corresponding to $K$ nonempty clusters belongs to ${\rm RMet}^{=K}_n$ and
by the above argument is a vertex of ${\rm RMet}^K_n$. Moreover, ${\rm RMet}^{=K}_n$ is a facet of ${\rm RMet}^K_n$; it then follows that all
such ratio-cut vectors corresponding to $K$ nonempty clusters are vertices of ${\rm RMet}^{=K}_n$ as well.
\end{proof}

\section{A new LP relaxation for K-means clustering}
\label{sec:relax}

As we detailed in Section~\ref{sec:ratioCutPoly}, the polytope ${\rm RCut}^{=K}_n$ corresponds
to a projection of the convex hull of the feasible region of Problem~\eqref{Kmeans2} defined by~\eqref{projection}.
Hence, to obtain an LP relaxation for K-means clustering, first, we outer-approximate the polytope
${\rm RCut}^{=K}_n$ by the polytope ${\rm RMet}^{=K}_n$. Next, we
introduce additional variables $X_{ii} := 1 - \sum_{j \in [n] \setminus \{i\}} {X_{ij}}$ for all $i \in [n]$
and let $X_{ji} = X_{ij}$ for all $1 \leq i < j \leq n$. Let $t \in \{2,\ldots, K\}$; for each $i \in [n]$ define
$\S^t_i :=\{ S \subseteq [n] \setminus\{i\}: 2 \leq |S| \leq t\}$.
It then follows that an LP relaxation of K-means clustering is given by:
\begin{align}
\label{lp:pK}
\tag{LPK}
\min  \quad  & \sum_{i,j \in [n]} {d_{ij} X_{ij}}\nonumber\\
\text{s.t.} \quad  & \Tr(X) = K,\label{eq1k}\\
& \sum_{j=1}^n X_{ij} = 1, \quad \forall i \in [n],\label{eq2k}\\
& \sum_{j \in S} {X_{ij}} \leq X_{ii} + \sum_{j, k \in S: j < k} {X_{jk}}, \quad \forall i \in [n], \; \forall S \in \S^t_i,\label{eq3k}\\
& X_{ij} \geq 0, \quad \forall 1 \leq i < j \leq n. \label{eq4k}
\end{align}
By the proof of Proposition~\ref{trivialFacet}, if $K =2$,
inequalities~\eqref{eq4k} are implied by equalities~\eqref{eq1k},~\eqref{eq2k} and inequalities~\eqref{eq3k}. However, for $K \geq 3$, these inequalities are facet-defining at hence are present in Problem~\eqref{lp:pK}.

Clearly, letting $t = K$ results in the strongest LP relaxation. However, it is important to note that system~\eqref{eq3k} contains $\Theta(n^{t+1})$ inequalities and hence, for large $t$, Problem~\eqref{lp:pK} is too expensive to solve.
Indeed,  even for $t = 2$, the above LP is expensive to solve for $n \geq 200$. To address this issue, a common approach is
to devise a cutting plane generation scheme together with the dual Simplex algorithm to solve the LP in an iterative manner (see for example~\cite{LisRen03}). Given a fixed $t$, let us consider the separation problem over inequalities~\eqref{eq3k}. For a fixed $i \in [n]$, these inequalities can be written as $\sum_{j, k \in S: j < k} {X_{jk}}  \geq  \sum_{j \in S} {X_{ij}}- X_{ii}$.
Hence, the separation problem over these inequalities is quite similar to the separation problem over \emph{clique inequalities} and therefore similar heuristics and acceleration techniques can be used (see for example~\cite{clique19}). We should also remark that
inequalities of the form~\eqref{eq3k} corresponding to subsets $S$ of smaller cardinality are \emph{sparser} (\ie contain fewer variables with nonzero coefficients) than those of the larger cardinality, a property that is highly desirable from a computational perspective.
A careful selection of $t$ and designing of an efficient cutting plane algorithm requires a systematic computational study and is a subject of future research.
However, as we demonstrate in Section~\ref{sec:numerics}, the LP relaxation with $t = 2$ outperforms the SDP relaxation, even for $K=5$ clusters.

\begin{remark}
Recall that Problem~\eqref{badLP} is the existing LP relaxation for K-means clustering~\cite{Awaetal15}.
To show that the feasible region of Problem~\eqref{lp:pK} is contained in the feasible region of this LP, it suffices to prove that inequalities $X_{ij} \leq X_{ii}$ for all $i \neq j \in [n]$ are implied by system~\eqref{eq1k}-\eqref{eq4k}. Without loss of generality
consider $X_{12} \leq X_{11}$. Consider the following inequalities and equalities all of
which are present in system~\eqref{eq1k}-\eqref{eq4k}:
\begin{itemize}
\item [(i)] $X_{12} + X_{13} \leq X_{11} + X_{23}$,
\item [(ii)] $X_{12} + X_{23} \leq X_{22} + X_{13}$,
\item [(iii)] $X_{12} + X_{1j} \leq X_{11} + X_{2j}$, for all $j \in \{4,\cdots,n\}$,
\item [(iv)] $\sum_{j \in [n]}{X_{1j}} = 1$,
\item [(v)] $\sum_{j \in [n]}{X_{2j}} = 1$.
\end{itemize}
Multiplying inequality~(i) by +2, inequality~(ii) by +1, each inequality of type~(iii) by +1, equality~(iv) by -1, equality~(v) by +1
and adding all resulting inequalities and equalities yields $n X_{12} \leq n X_{11}$ and this completes the argument.
\end{remark}

\subsection{Optimality of the planted clusters}

We now focus on the case with two clusters and obtain a sufficient condition
under which the ratio-cut vector corresponding to a planted clustering is an optimal solution of the LP relaxation.
As we discussed before, even with only two clusters, K-means clustering is NP-hard~\cite{BodGeo19}.
The LP relaxation of K-means clustering for $K=2$ is given by:
\begin{align}
\tag{LP2}
\label{lp:pK2}
\min  \quad & \sum_{i,j \in [n]} {d_{ij} X_{ij}} \nonumber\\
\text{s.t.} \quad & \Tr(X) = 2, \label{eq1}\\
& \sum_{j=1}^n X_{ij} = 1, \quad \forall 1\leq i\leq n, \label{eq2}\\
& X_{ij}+X_{ik} \leq X_{ii} + X_{jk}, \quad \forall i \neq j \neq k \in [n], \; j < k, \label{eq3}
\end{align}
where as before we let $X_{ji} = X_{ij}$ for all $1 \leq i < j \leq n$.
We start by constructing the dual of Problem~\eqref{lp:pK2};
define dual variables $\omega$ associated with~\eqref{eq1}, $\mu_i$, $i \in [n]$ associated  with~\eqref{eq2},
and $\lambda_{ijk}$ for all $(i,j,k) \in \Omega := \{(i,j,k): i \neq j \neq k \in [n], \; j < k\}$ associated  with~\eqref{eq3}. It then follows
that the dual of Problem~\eqref{lp:pK2} is given by
\begin{align}
\max \quad  & -(2\omega + \sum_{i \in [n]}{ \mu_i}) \nonumber\\
\text{s.t.} \quad & \mu_i + \mu_j + \sum_{k \in [n] \setminus \{i,j\}} {(\lambda_{ijk} + \lambda_{jik}-\lambda_{kij})} + 2 d_{ij} = 0, \;
\forall 1 \leq i < j \leq n, \label{d1}\\
&\omega + \mu_i -\sum_{\substack{j,k \in [n] \setminus \{i\}:\\ j < k}}{\lambda_{ijk}} = 0, \; \forall i \in [n],\label{d2}\\
& \lambda_{ijk} \geq 0, \quad \forall (i,j,k) \in \Omega, \nonumber
\end{align}
where we let $\lambda_{ikj} =\lambda_{ijk}$ for all $(i,j,k) \in \Omega$. Now consider the following planted model: suppose that $n$ is even, the first half of the points are in the first cluster and the
second half are in the second cluster; define $\C_1 :=\{1,\ldots,n/2\}$
and $\C_2 := \{n/2+1,\ldots,n\}$. Then the ratio-cut vector associated with this planted clustering is
given by:  $\bar X_{ij} = \frac{2}{n}$ for all $i < j \in \C_1$
and for all $i < j \in \C_2$ and $\bar X_{ij} = 0$, otherwise.
We would like to obtain conditions under which $\bar X$
is an optimal solution of Problem~\eqref{lp:pK2}.
To this end, it suffices to find a dual feasible point
$(\bar \lambda, \bar \mu, \bar \omega)$ for which strong duality is attained:
\begin{equation}\label{str}
2 \bar \omega + \sum_{i \in [n]}{ \bar \mu_i} = -\frac{4}{n}
\Bigg(\sum_{i,j\in \C_1: i < j} {d_{ij}} + \sum_{i,j \in \C_2: i <j} {d_{ij}} \Bigg).
\end{equation}
For notational simplicity,  for every $A \subseteq [n]$ and $f:A\to \R$, we define
$$
\underset{i \in A}{\dashsum}{f(i)} := \frac 1{|A|}\sum_{i \in A}{f(i)}.
$$
Moreover,
for each $i \in \C_l$, $l \in \{1,2\}$, let us define
$\din_i := \underset{j\in \C_l}{\dashsum}{d_{ij}}$,
and
$\dout_i := \underset{j \in [n]\setminus \C_l}{\dashsum}{d_{ij}}$.

We now present a sufficient condition for the optimality of the planted clusters.
%

\begin{theorem}\label{LPopt}
Define
\begin{equation}\label{cond3}
\eta := \frac12\Big(\underset{i \in \C_1} {\dashsum}{\din_i} + \underset{i \in \C_2}{\dashsum} {\din_i}\Big).
\end{equation}
Then the ratio-cut vector corresponding to the planted clusters is an optimal solution of Problem~\eqref{lp:pK2} if
for all  $i,j \in \C_l$,  $l \in \{1,2\}$, we have
%
\begin{equation}\label{easyC}
\underset{k \in [n] \setminus \C_l}{\dashsum}{\min\{d_{ik} + \din_j , d_{jk} + \din_i\}} -d_{ij}\geq  \eta.
\end{equation}
\end{theorem}


\begin{proof}
By complementary slackness $\bar\lambda_{ijk} = 0$ if $i \in \C_1$ and $j,k \in \C_2$
or if $i \in \C_2$ and $j,k \in \C_1$. Substituting in~\eqref{d1} yields:
\begin{equation}\label{d11}
\bar \mu_i + \bar \mu_j + \sum_{k \notin \C_l} {(\bar\lambda_{ijk} + \bar\lambda_{jik})}
+ \sum_{k \in \C_l \setminus \{i,j\}} {(\bar\lambda_{ijk} + \bar\lambda_{jik}-\bar\lambda_{kij})}
+ 2 d_{ij} = 0,
\end{equation}
for every $i < j$ such that $i,j \in \C_l$, $l \in \{1,2\}$ and
\begin{equation}\label{d12}
\bar\mu_i + \bar\mu_j + \sum_{k \in \C_1 \setminus \{i\}} {(\bar\lambda_{ijk}-\bar\lambda_{kij})} + \sum_{k \in \C_2 \setminus \{j\}} {(\bar\lambda_{jik}-\bar\lambda_{kij})} + 2 d_{ij} = 0,
\end{equation}
for every $i \in \C_1$, $j \in \C_2$.
Moreover, equation~\eqref{d2} simplifies to
\begin{equation}\label{dp21}
\bar\omega + \bar\mu_i -\sum_{\substack{j \in \C_l \setminus \{i\},\\ k \notin \C_l}}{\bar\lambda_{ijk}} -\sum_{j <k \in \C_l \setminus \{i\}}{\bar\lambda_{ijk}}= 0,
\end{equation}
for each $i \in \C_l$, $l \in \{1,2\}$. Now for each $i, j \in \C_l$ and $k \notin \C_l$, $l \in \{1,2\}$, let
\begin{equation}\label{l1}
\bar\lambda_{ijk} - \bar\lambda_{jik} = \frac{d_{jk}-d_{ik}}{n/2}+ \frac{\din_{i}-\din_{j}}{n/2}.
\end{equation}
Substituting~\eqref{l1} in~\eqref{d12} yields:
\begin{equation}\label{i1}
\bar\mu_i + \bar\mu_j + \dout_i + \dout_j + \din_i + \din_j -\underset{k \in \C_1}{\dashsum}{\din_k}-\underset{k \in \C_2}{\dashsum}{\din_k}= 0, \quad \forall i \in \C_1, \; j \in \C_2.
\end{equation}
To satisfy~\eqref{i1}, let
\begin{equation}\label{i2n}
\bar\mu_i = - \din_i -\dout_i  + \eta, \quad \forall i \in [n].
\end{equation}
%
%
where $\eta$ is defined by~\eqref{cond3}.
Substituting~\eqref{i2n} in~\eqref{str} we obtain
\begin{equation}\label{omega}
\bar\omega = \frac{1}{2} \sum_{i \in [n]} {(\dout_{i}-\din_{i}}).
\end{equation}
Utilizing~\eqref{l1} and~\eqref{i2n}, equation~\eqref{d11} can be equivalently written as
\begin{equation}\label{after}
\sum_{k \in \C_2} {\bar\lambda_{ijk}} + \frac{1}{2}\sum_{k \in \C_1 \setminus \{i,j\}} {(\bar\lambda_{ijk} + \bar\lambda_{jik}-\bar\lambda_{kij})}= \din_i - d_{ij} + \dout_j  - \eta,
\end{equation}
for any $(i, j) \in \C_1$.
By~\eqref{i2n} and~\eqref{omega}, for each $i \in \C_1$, equation~\eqref{dp21} simplifies to
\begin{equation}\label{elim}
\sum_{\substack{j \in \C_1 \setminus \{i\},\\ k \in \C_2}} {\bar\lambda_{ijk}} + \sum_{j <k \in \C_1 \setminus \{i\}}{\bar\lambda_{ijk}}=
 \frac{1}{2} \sum_{j \in [n]} {(\dout_{j}-\din_{j}})-\dout_i - \din_i +\eta.
\end{equation}
We now obtain a  set of conditions under which system~\eqref{elim} is implied by equalities~\eqref{after}.
Firstly, it can be checked that
$$\sum_{j \neq k \in \C_1 \setminus \{i\}} {(\bar\lambda_{ijk} + \bar\lambda_{jik}-\bar\lambda_{kij})} = 2 \sum_{j <k \in \C_1 \setminus \{i\}}{\bar\lambda_{ijk}}.$$
Hence, for each $i \in \C_1$
it suffices to have
\begin{equation*}
\sum_{j \in \C_1 \setminus \{i\}}{\Big(\din_i - d_{ij} + \dout_j  -\eta \Big)} = \frac{1}{2} \sum_{j \in [n]} {(\dout_{j}-\din_{j}})-\dout_i - \din_i +\eta,
\end{equation*}
whose validity can be verified by a simple calculation.
%
Hence, to find a dual certificate, it suffices to find nonnegative $\bar \lambda_{ijk}$ satisfying equalities~\eqref{d11};
that is, for each $(i ,j) \in \C_1$, we should find $\bar \lambda_{ijk}$ satisfying the following system
\begin{align*}
& \sum_{k \in \C_2} {(\bar\lambda_{ijk} + \bar\lambda_{jik})} + \!\!\sum_{k \in \C_1 \setminus \{i,j\}} \!\!{(\bar\lambda_{ijk} + \bar\lambda_{jik}-\bar\lambda_{kij})}= \dout_i + \dout_j - 2 d_{ij} + \din_i + \din_j -2 \eta \\
& \bar\lambda_{ijk} \geq 0, \; \bar\lambda_{jik} \geq 0, \quad \forall k \in \C_2 \\
& \bar\lambda_{ijk} - \bar\lambda_{jik} = \frac{d_{jk} - d_{ik}}{n/2} + \frac{\din_i-\din_j}{n/2} , \quad \forall k \in \C_2\\
& \bar\lambda_{ijk} \geq 0,  \; \bar\lambda_{jik} \geq 0, \; \bar\lambda_{kij} \geq 0, \quad \forall  k \in \C_1 \setminus\{i,j\}.
\end{align*}
By letting $\bar\lambda_{ijk} \geq \max \Big\{0, \frac{d_{jk} - d_{ik}}{n/2} + \frac{\din_i-\din_j}{n/2} \Big\}$
and $\bar\lambda_{jik} \geq  \max \Big\{0, \frac{d_{ik} - d_{jk}}{n/2} + \frac{\din_j-\din_i}{n/2} \Big\}$ for all $k \in \C_2$,
it follows that the above system has a feasible solution, if
\begin{equation}\label{mess}
  \frac{1}{2}\sum_{k \in \C_1 \setminus \{i,j\}} {(\bar\lambda_{ijk} + \bar\lambda_{jik}-\bar\lambda_{kij})} \leq
  \underset{k \in \C_2}{\dashsum}{\min\{d_{ik} + \din_j , d_{jk} + \din_i\}} -d_{ij} - \eta,
\end{equation}
for all $i,j \in \C_1$ together with nonnegativity of the remaining multipliers.

By letting $\bar\lambda_{ijk} = 0$ for all $i,j,k \in \C_l$, $l \in \{1,2\}$, inequality~\eqref{mess} simplifies to
condition~\eqref{easyC} and this completes the proof.
\end{proof}

As we demonstrate in Section~\ref{sec:sbm}, for the SBM, condition~\eqref{easyC} leads to an overly conservative estimate for the minimum separation distance between cluster centers. We are hoping that the theoretical analysis presented in this paper serves as a
starting point for deriving more realistic recovery guarantees for the proposed LP relaxation. Obtaining recovery guarantees for $K \geq 3$ is left for future research as well.

\section{Recovery under the stochastic ball model}
\label{sec:sbm}

In this section, we consider a popular generative model for K-means clustering often referred to as the~\emph{stochastic ball model}
in the literature. This random model is defined as follows: let $\{\gamma^k\}_{k \in [K]}$ be ball centers in $\R^m$.
For each $k$, draw i.i.d. vectors $\{y^{k,i}\}_{i=1}^n$ from some rotation-invariant distribution supported on the unit ball.
The points in cluster $k$ are then taken to be $x^{k,i} := y^{k,i} + \gamma^k$. Moreover, we define $\Delta := \min_{k \neq l \in [K]} ||\gamma^k-\gamma^l||_2$.

Henceforth, we focus on the case of two clusters;
throughout this section, whenever we say \emph{with high probability}, we imply with probability tending to one as $n$ tends to infinity.
We are interested in the following question: what is the minimum separation distance $\Delta$ required for the LP relaxation to recover
the planted clusters with high probability? Before proceeding further with addressing this question, we first establish a recovery threshold for K-means clustering under the SBM. This threshold then serves as a recovery limit for any convex relaxation of K-means clustering. In the following, we denote by $\{e_i\}_{i=1}^m$ the standard basis for $\R^m$. Moreover, for any $k \in [m]$, we denote by $\H^k$ the $k$-dimensional Hausdorff measure.

\subsection{Recovery for K-means clustering}
In~\cite{BodGeo19}, the authors show that if the points are uniformly generated on two $m$-dimensional touching spheres for
some $m \geq 3$, in the continuum limit, the K-means clustering problem identifies the two individual spheres as clusters.
The goal of this section is to show that a similar recovery result is valid for the SBM.

We start by introducing some notation.
We denote by $B(x,r)$ the closed $m$-dimensional ball centered at $x$ with radius $r$. For a set $A\subset \R^m$, we denote by $\overline{A}$ the closure of $A$ and by $\partial A$ the boundary of $A$.
Given a Borel measure $\rho$ on $\mathbb R^m$ with support $S$
and a Borel-measurable subset $S_1 \subset S$ with complement $S_2 = S \setminus S_1$,
the {\it mean squared error} associated with the partition $\{S_1,S_2\}$ of $S$
is
$$
   \mathcal{E}_S (S_1) =  \min_{c\in \mathbb{R}^m} \int_{S_1} \|x-c\|^2 d \rho(x) +  \min_{d \in \mathbb{R}^m}\int_{S_2} \|x-d\|^2 d \rho(x).
$$
For every  Borel subset $A\subset \R^m$ and every $k \in [m]$, we define the measure $\H^k\res A$ as follows:
$$\H^k\res A(B)=\H^k(A\cap B), \qquad \mbox{for every Borel subset $B \subset \R^m$}.$$
For every  Borel-measurable subset $A\subset \R^m$, we denote by
$b(A):=\int_A x dx$
the barycenter of $A$. It is easy to check that, if $A$ is a $k$-dimensional smooth set and $\H^k\res A$ is a finite non-zero measure, than $b(A)$ is the only minimizer of the function  $y\in \R^m \mapsto \int \|x-y\|^2 d\H^k\res A$.
In this section we prove the following result:
\begin{theorem} \label{thm:intero}
For any $m\geq 3$, let $S:=B(-e_1,1)\cup  B(e_1,1)$ and $\rho:= \H^{m}\res B(-e_1,1)+ \H^{m}\res B(e_1,1)$. Then, up to a set of zero Lebesgue measure, the partition $\{B(-e_1,1),$ $ B(e_1,1)\}$ of $S$ is the unique minimizer of the mean squared error.
\end{theorem}

In~\cite{BodGeo19}, the authors prove Theorem \ref{thm:intero} in the case where
$\rho$ is the surface measure for the union of two touching spheres, \ie $ \rho = \H^{m-1}\res  \partial B(-e_1,1)+ \H^{m-1}\res  \partial B(e_1,1)$.
To this end, they first prove that an optimal partition is given by a
separating hyperplane that is orthogonal to the symmetry axis (Lemma~3.5 and Theorem~2.2 in~\cite{BodGeo19}).
Subsequently,  they examine
the offset of the optimal separating hyperplane. More precisely they show the following:


\begin{proposition} \label{thm:uniqueminimum} [Theorems~2.4 and~2.5 in~\cite{BodGeo19}]
For any $m\geq 3$, let $S:=\partial B(-e_1,1)\cup \partial B(e_1,1)$ and  $\rho := \H^{m-1} \res \partial B(-e_1,1)+ \H^{m-1}\res  \partial B(e_1,1)$. Then the function $a\in \R \mapsto \mathcal{E}_S(\{x \in S: x_1 \le -a \})$ attains a minimum at $a=0$, and this minimum is unique.
\end{proposition}

\begin{remark}\label{rem1}
Proposition \ref{thm:uniqueminimum} is invariant under scaling. In particular if for any $r>0$ we define $S:=\partial B(-re_1,r)\cup \partial B(re_1,r)$ and $\rho:= \H^{m-1}\res \partial B(-re_1,r)+ \H^{m-1}\res $ $ \partial B(re_1,r)$, then the function $a\in \R \mapsto \mathcal{E}_S(\{x \in S: x_1 \le -a \})$ attains a minimum at $a=0$, and this minimum is unique. Moreover, in~\cite{BodGeo19}, the authors renormalize $\rho$ to get a probability measure, but this changes $\mathcal{E}_S (S_1)$ just by a constant factor.
\end{remark}

Since the proofs of Lemma~3.5 and Theorem~2.2 in~\cite{BodGeo19} can be repeated verbatim for balls, in order to prove Theorem~\ref{thm:intero}, we just need to prove the following analogous result to Proposition \ref{thm:uniqueminimum}:

\begin{proposition} \label{thm:uniqueminimumballs}
For any $m\geq 3$, let $S:=B(-e_1,1)\cup  B(e_1,1)$ and $\rho:= \H^{m}\res$ $ B(-e_1,1)+ \H^{m}\res B(e_1,1)$. Then the function
 $$a\in \R \mapsto F(a ) :=\mathcal{E}_S(\{x \in S: x_1 \le -a \})$$
  attains a minimum at $a=0$, and this minimum is unique.
\end{proposition}
\begin{proof}
Assume by contradiction there exists $a\neq 0$ such that $F(a )\leq F(0)$. By symmetry, we can assume $a> 0$.
Define $S_1:=\{x \in S: x_1 \le -a \}$, $S_1^r:=\{x \in\partial B(-e_1,r) : x_1 \le -a \}$, $S_2^r:=\{x \in\partial B(-e_1,r) : x_1 \ge -a \}\cup \partial B(e_1,r)$, and as before, we let
$S_2 = S \setminus S_1$. Then, by Coarea formula~(see Chapter~2 of~\cite{TonioBook})
\begin{equation*}
\begin{split}
\int_0^1&\int_{\partial B(-e_1,r)} \|x+e_1\|^2 d\H^{m-1}(x)dr +  \int_0^1\int_{\partial B(e_1,r)} \|x-e_1\|^2 d\H^{m-1}(x)dr=F(0)\\
&\geq  F(a )=\int_{S_1} \|x-b(S_1)\|^2 d x +  \int_{S_2} \|x-b(S_2)\|^2 dx\\
&=\int_0^1\int_{S_1^r} \|x-b(S_1)\|^2 d\H^{m-1}(x)dr +  \int_0^1\int_{S_2^r} \|x-b(S_2)\|^2 d\H^{m-1}(x)dr.
\end{split}
\end{equation*}
We deduce that there exists $r\in (0,1)$, such that
\begin{equation}\label{c2}
\begin{split}
\int_{\partial B(-e_1,r)}& \|x+e_1\|^2 d\H^{m-1}(x) + \int_{\partial B(e_1,r)} \|x-e_1\|^2 d\H^{m-1}(x)\\
&\geq \int_{S_1^r} \|x-b(S_1)\|^2 d\H^{m-1}(x) +  \int_{S_2^r} \|x-b(S_2)\|^2 d\H^{m-1}(x).
\end{split}
\end{equation}
Define $S_3^r:=\{x \in\partial B(-e_1,r) : x_1 \ge -a \}\cup \partial B((2r-1)e_1,r)$. It then follows that
\begin{equation}\label{c1}
\begin{split}
&\int_{S_1^r} \|x-b(S_1)\|^2 d\H^{m-1}(x) +  \int_{S_2^r} \|x-b(S_2)\|^2 d\H^{m-1}(x)\\
&\geq \int_{S_1^r} \|x-b(S_1^r)\|^2 d\H^{m-1}(x) +  \int_{S_3^r} \|x-b(S_3^r)\|^2 d\H^{m-1}(x) {\color{red} }\\
&> \int_{\partial B(-e_1,r)} \|x+e_1\|^2 d\H^{m-1}(x) +  \int_{\partial B((2r-1)e_1,r)} \|x-(2r-1)e_1\|^2 d\H^{m-1}(x)\\
&=\int_{\partial B(-e_1,r)} \|x+e_1\|^2 d\H^{m-1}(x) + \int_{\partial B(e_1,r)} \|x-e_1\|^2 d\H^{m-1}(x),
\end{split}
\end{equation}
where the first inequality follows from the definition of the barycenter and the second inequality follows from Proposition \ref{thm:uniqueminimum} and Remark \ref{rem1}.
Combining \eqref{c2} with \eqref{c1}, we get the desired contradiction.
\end{proof}

\begin{remark}\label{Lowdimension}
It is well-known that in dimension one, for both spheres and balls,
K-means clustering recovers the planted clusters with high probability only if $\Delta > 1 + \sqrt{3}$~(see for example~\cite{IduMixPetVil17,BodGeo19}).
As we detail in the next section $\Delta > 1 + \sqrt{3}$ is also sufficient for recovery with high probability, hence settling the question in dimension one.
In~\cite{BodGeo19} the authors show that for two touching spheres in dimension two, the minimum of  $F(a ) =\mathcal{E}_S(\{x \in S: x_1 \le -a \})$ is attained at a point $a \neq 0$. The authors of~\cite{IduMixPetVil17} numerically verify that
a similar result holds for the SBM in dimension two.
To date, the recovery threshold in dimension two, for both spheres and balls, remains an open question.
\end{remark}

\subsection{Recovery for the LP relaxation}

In this section, we obtain a recovery guarantee for the proposed LP relaxation under the SBM.
Namely, we prove that our deterministic optimality condition given by inequality~\eqref{easyC} implies that
Problem~\eqref{lp:pK2} recovers the planted
clusters with high probability, provided that $\Delta > 1 + \sqrt{3} \approx 2.73$.
This is a significant improvement compared to the only existing
recovery result for an LP relaxation of K-means clustering stating that recovery is possible
with high probability if and only if $\Delta > 4$~\cite{Awaetal15}.
Moreover, by Remark~\ref{Lowdimension}, this sufficient condition is tight for $m=1$ and implies that in dimension one, the NP-hard K-means clustering problems recovers the planted clusters with high probability if and only if $\Delta > 1 + \sqrt{3}$.

In the remainder of this section, for an event $A$, we denote by $\prob(A)$ the probability of $A$.
We denote by $\avg[Y]$ the expected value of a random variable $Y$. In case of a multivariate random variable $X_{ij}$, the conditional expected value in $j$, with $i$ fixed, will be denoted either with $\avg_i[X]$ or with $\avg^j[X]$.
To prove the next theorem, we make use of two technical lemmas; to streamline the presentation, these lemmas are given in Section~\ref{sec:proofs}.

\begin{theorem}\label{recovery}
Let $K = 2$ and suppose that the points are generated according to the SBM.
Then Problem~\eqref{lp:pK2} recovers the planted clusters with high probability, if $\Delta > 1+ \sqrt{3}$.
\end{theorem}
\begin{proof}
To prove the statement, we need to show that for $\Delta > 1+ \sqrt{3}$, with high probability
the ratio-cut vector corresponding to the planted clusters is the unique optimal solution of Problem~\eqref{lp:pK2}.

To prove uniqueness, notice that the solution
to the LP is not unique only if the objective function coefficient vector $d = \{d_{ij}\}_{1 \leq i < j \leq n}$ is orthogonal to an edge of the polytope ${\rm RMet^{=2}_n}$.
The objective function coefficient vector is generated from a probability distribution which is absolutely continuous with respect to the Lebesgue measure $\H^{\binom{n}{2}}$ in $\R^{\binom{n}{2}}$.
The set of all ``bad'' directions however is the union of finitely many $\binom{n}{2}-1$-dimensional subspaces and hence is a zero $\H^{\binom{n}{2}}$-measure set. Hence any optimal solution is unique with probability one.

We now address the question of optimality of the planted clusters under the SBM. In particular, we show that the optimality condition~\eqref{easyC} holds with high probability.
Namely, we show that, given $\epsilon > 0$ as defined in the statement of Lemma \ref{MasterOfProbability} (since $\Delta > 1+ \sqrt{3}$), we have
\begin{equation*}
\begin{split}
\prob \Big (\bigcap_{i,j\in \C_1}\Big\{d_{ij} &+\eta-\underset{k \in \C_2}{\dashsum}{\min\{d_{ik} + \din_j , d_{jk} + \din_i\}} \leq 0\Big\}\Big ) \\
&\geq 1-\Big(4e^{-2\binom{n/2}{2}\epsilon^2/16}+ne^{-n \epsilon^2/32}+2\binom{n/2}{2}e^{-2n \epsilon^2/128}\Big ).
\end{split}
\end{equation*}
We first observe that
\begin{equation}\label{usef}
\begin{split}
\prob &  (\{ |\eta-\avg[\eta]|\geq \epsilon\}) \\
&= \prob\Big(\Big |\underset{i,j \in \C_1}{\dashsum}d_{ij} -  \avg\Big[\underset{i,j \in \C_1}{\dashsum}d_{ij}\Big] +\avg\Big[\underset{i,j \in \C_2}{\dashsum} d_{ij}\Big]- \underset{i,j \in \C_2}{\dashsum} d_{ij}\Big | \geq 2\epsilon \Big) \\
&\leq \prob\Big (\Big\{\Big |\underset{i,j \in \C_1}{\dashsum}d_{ij} -  \avg\Big[\underset{i,j \in \C_1}{\dashsum}d_{ij}\Big]\Big|\geq \epsilon \Big\}  \cup \Big\{\Big |\avg\Big[\underset{i,j \in \C_2}{\dashsum} d_{ij}\Big]- \underset{i,j \in \C_2}{\dashsum} d_{ij}\Big | \geq \epsilon\Big\}\Big ) \\
&\leq  \prob\Big (\Big |\underset{i,j \in \C_1}{\dashsum}d_{ij} -  \avg\Big[\underset{i,j \in \C_1}{\dashsum}d_{ij}\Big] \Big|\geq \epsilon\Big)+ \prob\Big (\Big |\avg\Big[\underset{i,j \in \C_2}{\dashsum} d_{ij}\Big]- \underset{i,j \in \C_2}{\dashsum} d_{ij}\Big | \geq \epsilon\Big ) \\
&\leq 4e^{-2\binom{n/2}{2}\epsilon^2/16}.
\end{split}
\end{equation}
The first inequality holds by set inclusion and the third inequality follows from Hoeffding's inequality (see for example Theorem 2.2.6 in~\cite{VerBookHDP}), since
$d_{ij}$, $i,j \in \C_l$ are i.i.d. random variables for every $l \in \{1,2\}$ and $d_{ij} \in [0,4]$.

For notational simplicity, let us denote
$$t_{ij}:=\avg^k\Big[\underset{k \in \C_2}{\dashsum}{\min\{d_{ik} + \avg_j[ \din_j] , d_{jk} + \avg_i[\din_i]\}}\Big].$$
We now observe that
\begin{equation}\label{3333}
\begin{split}
\prob \Big (\bigcup_{i,j\in \C_1}\Big\{\Big |t_{ij}&-\avg^k\Big[\underset{k \in \C_2}{\dashsum}{\min\{d_{ik} + \din_j , d_{jk} + \din_i\}}\Big]\Big | \geq \epsilon\Big \}\Big ) \\
&\leq \prob \Big (\bigcup_{i\in \C_1}\Big\{\Big | \din_i- \avg_i[ \din_i]\Big |\geq \epsilon/2\Big\}\Big )\leq ne^{-n \epsilon^2/32},
\end{split}
\end{equation}
where the first inequality follows from the linearity of expectation and the second inequality follows from the application of Hoeffding's inequality and taking the union bound.  Combining the previous estimates, we conclude the claimed inequality:
\begin{equation*}
\begin{split}
&\prob \Big (\bigcap_{i,j\in \C_1}\Big\{ d_{ij} +\eta-\underset{k \in \C_2}{\dashsum}{\min\{d_{ik} + \din_j , d_{jk} + \din_i\}} \leq 0\Big\}\Big ) \\
&\geq \prob \Big(\bigcap_{i,j\in \C_1}\Big\{d_{ij} +\eta-d_{ij}-\avg[\eta]+t_{ij}-\underset{k \in \C_2}{\dashsum}{\min\{d_{ik} + \din_j , d_{jk} + \din_i\}} \leq 3\epsilon\Big\}\Big) \\
&\geq  \prob\Big (\Big\{\Big | \eta-\avg[\eta] \Big|<\epsilon\Big\}\cap \bigcap_{i,j\in \C_1}\Big\{\Big |t_{ij}-\avg^k\Big[\underset{k \in \C_2}{\dashsum}{\min\{d_{ik} + \din_j , d_{jk} + \din_i\}}\Big]\Big | <\epsilon\Big \} \\
&\cap\bigcap_{i,j\in \C_1}\Big\{\Big |\avg^k\Big[\underset{k \in \C_2}{\dashsum}{\min\{d_{ik} + \din_j , d_{jk} + \din_i\}}\Big]-\underset{k \in \C_2}{\dashsum}{\min\{d_{ik} + \din_j , d_{jk} + \din_i\}}\Big | <\epsilon \Big\} \Big ) \\
&\geq  1-\Big(4e^{-2\binom{n/2}{2}\epsilon^2/16}+ne^{-n \epsilon^2/32}+2\binom{n/2}{2}e^{-2n \epsilon^2/128}\Big ).
\end{split}
\end{equation*}
The first inequality follows from Lemma~\ref{MasterOfProbability}, since $\Delta > 1+ \sqrt{3}$; the second inequality holds by set inclusion; the
third inequality is obtained by taking the union bound, followed by the application of Hoeffding's inequality, inequalities~\eqref{usef}
and~\eqref{3333}.
\end{proof}

\section{Numerical Experiments}
\label{sec:numerics}
In this section, we conduct a preliminary numerical study to demonstrate the desirable numerical properties of the proposed LP relaxation defined by~\eqref{lp:pK}. A comprehensive computational study including the design and implementation of a separation algorithm for inequalities~\eqref{eq3k} is a subject of future research.
First, we compare the recovery properties of the proposed LP relaxation
versus the SDP relaxation defined by~\eqref{SDP}. To this end,
we set $m \in \{2,3\}$, $K \in \{2,3\}$ and generate the points in each cluster according
to the SBM.
We set $t = 2$ in the LP relaxation. As before we denote by $\Delta$ the minimum distance between the cluster centers.
For each fixed configuration $(m,K)$, we consider various values for $\Delta$;
namely, we set $\Delta \in [2:0.01: \bar \Delta]$,
where $\bar \Delta$ is set to a value at which recovery is clearly achieved for both algorithms.  For each fixed $\Delta$, we conduct 20
random trials. We count the number of times the optimization algorithm returns the planted clusters as the optimal solution;
dividing this number by total number of trials, we obtain the empirical rate of success. All experiments are performed on the {\tt NEOS} server~\cite{neos98}; LPs are solved with {\tt GAMS/CPLEX}~\cite{cplex} and SDPs are solved with {\tt GAMS/MOSEK}~\cite{mosek}.

Our results are depicted in Figure~\ref{figure1}. As can be seen from these graphs, in all configurations, the LP clearly outperforms the SDP in recovering the planted clusters. In particular, results for $K = 2$ suggest that our recovery guarantee of Section~\ref{sec:sbm} is excessively conservative.  In addition, it can be seen that the recovery threshold of the LP relaxation in dimension $m =3$ is better than the threshold in dimension $m =2$; this effect is not reflected in our recovery guarantee and
is a subject to future research.
We also remark that in all these experiments, the optimal solution of the LP is a partition matrix;
that is, even when the LP fails in recovering the planted clusters, its optimal solution is still a partition
matrix and hence optimal for the original nonconvex problem. This is in sharp contrast with the SDP relaxation whose solution is not a partition matrix whenever
it fails in recovering the planted clusters.

\begin{figure}[htbp]
 \centering
  \subfigure[$m=2$, $K=2$, $n = 100$]{\label{fig1e}\epsfig{figure=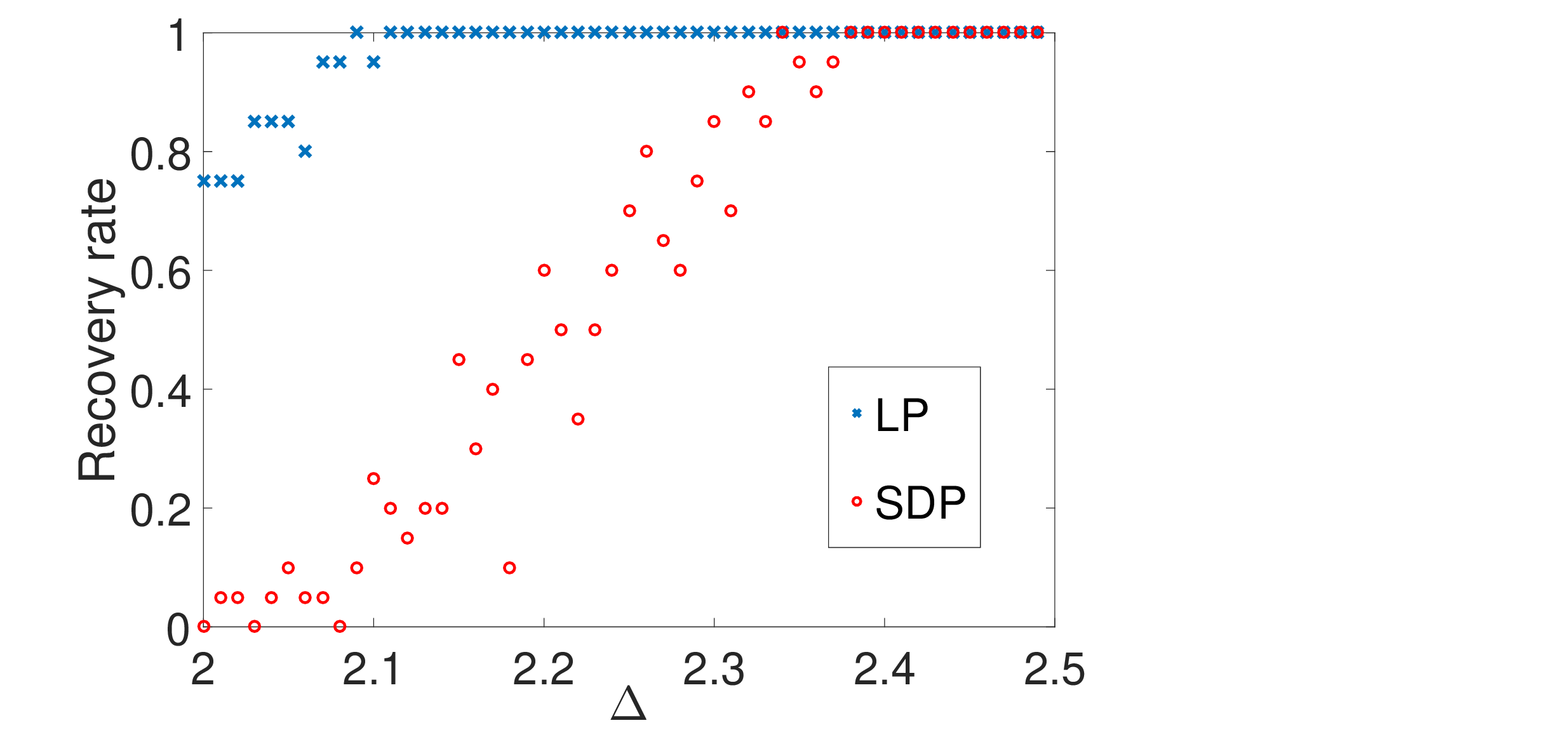, scale=0.21, trim=20mm 0mm 120mm 0mm, clip}}
 \subfigure [$m=2$, $K=3$, $n = 120$]{\label{fig1f}\epsfig{figure=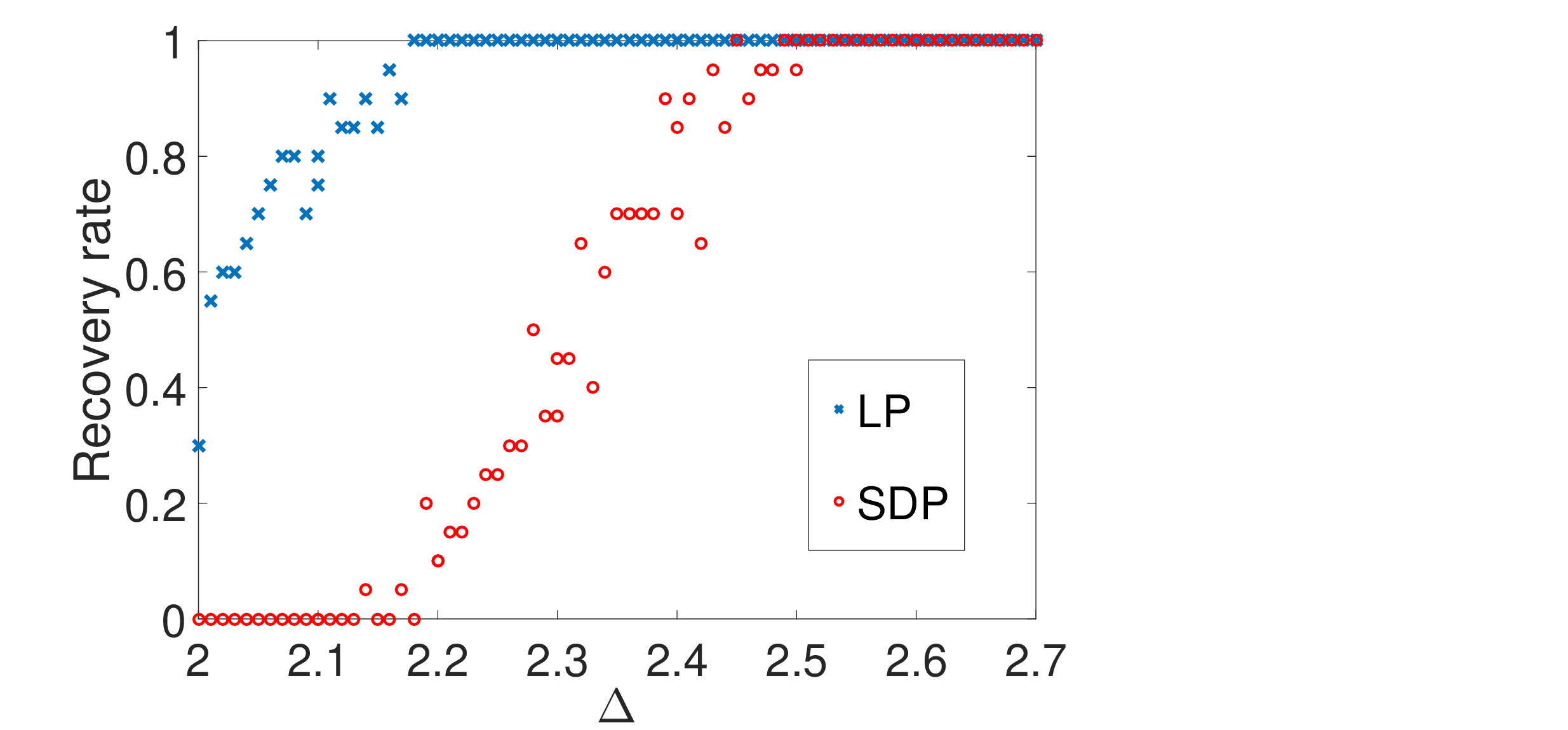, scale=0.21, trim=20mm 0mm 120mm 0mm, clip}}
 \subfigure [$m=3$, $K=2$, $n = 100$]{\label{fig1g}\epsfig{figure=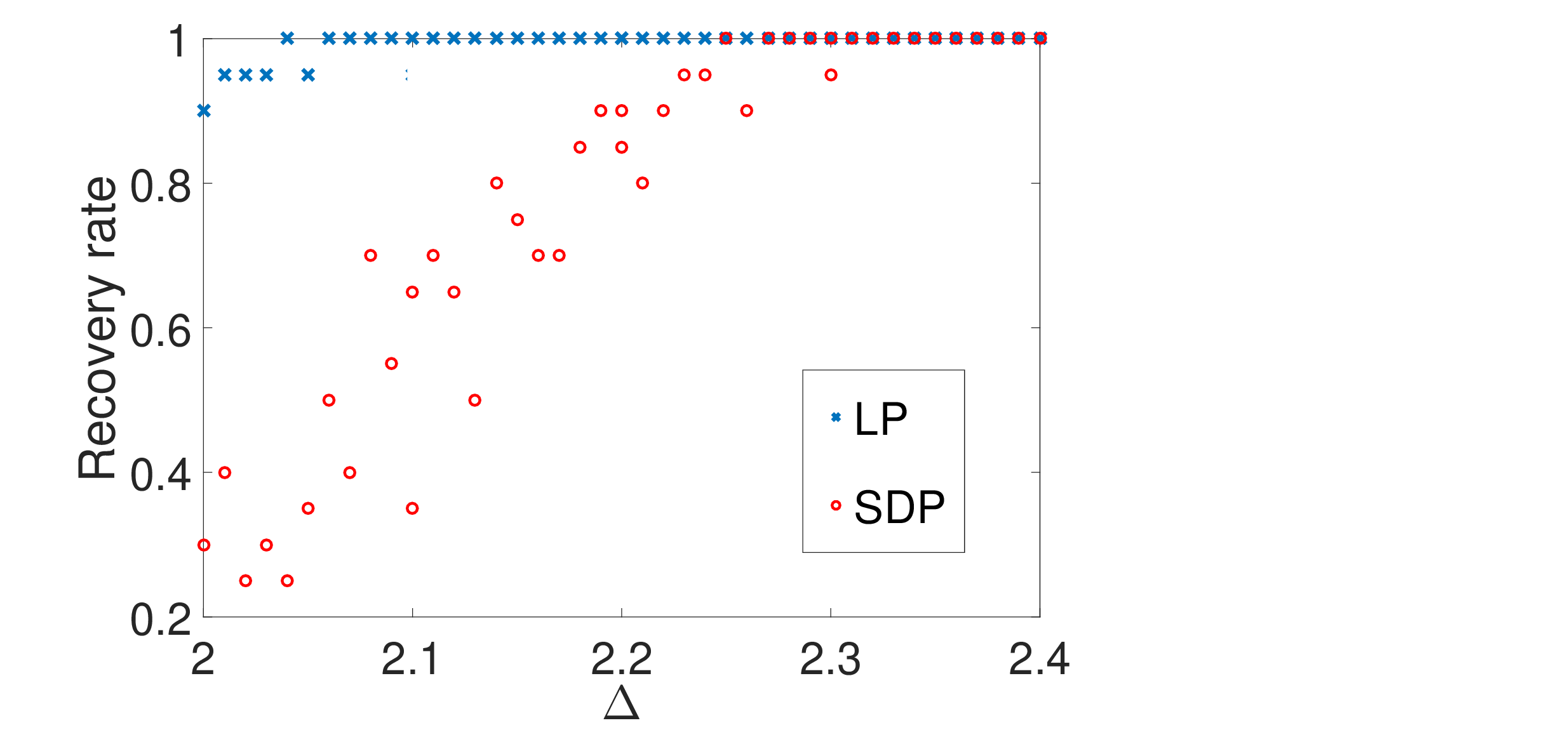, scale=0.21, trim=20mm 0mm 120mm 0mm, clip}}
 \subfigure [$m=3$, $K=3$, $n = 120$]{\label{fig1h}\epsfig{figure=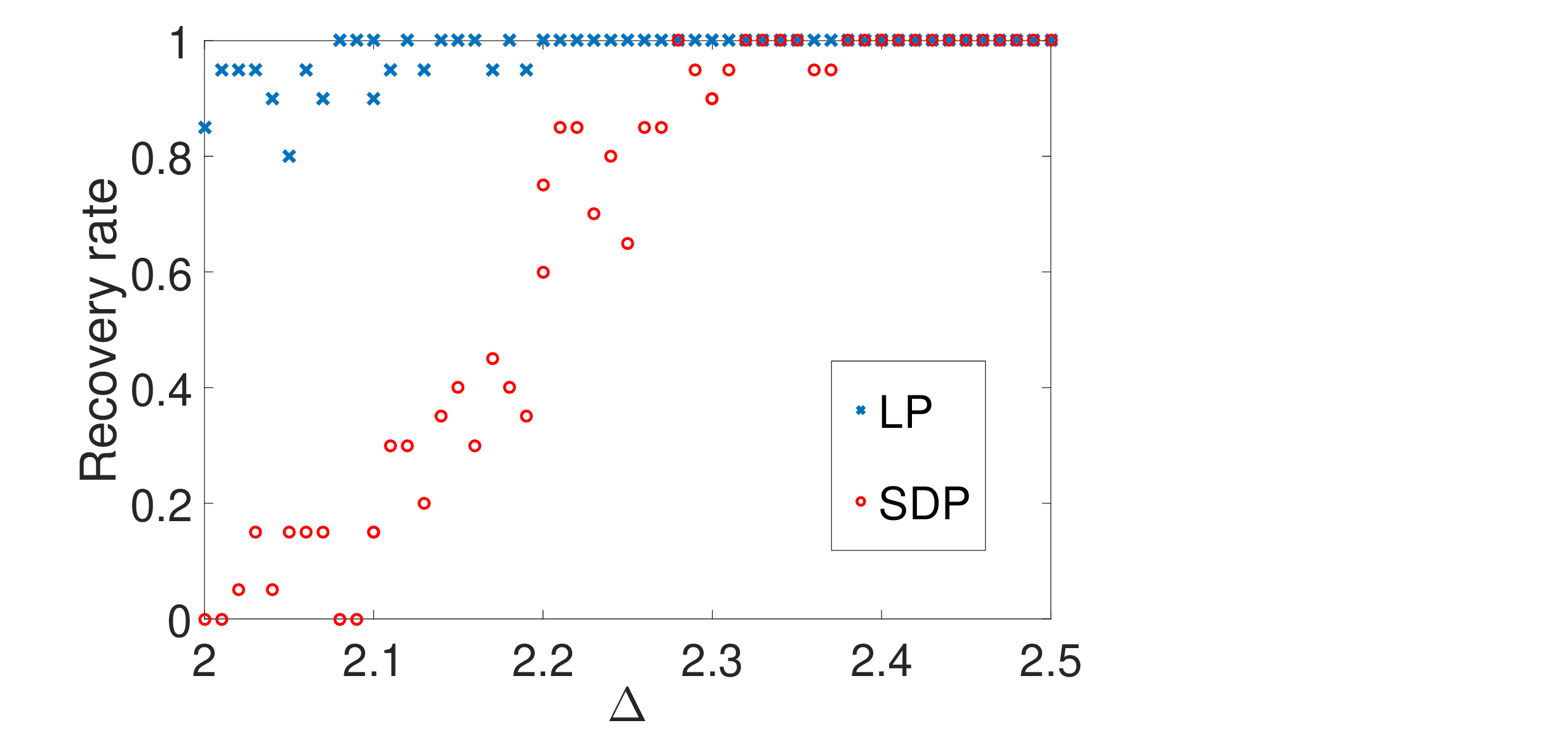, scale=0.21, trim=20mm 0mm 120mm 0mm, clip}}
 \caption{The empirical probability of success of the
LP versus the SDP in recovering the planted clusters when the points are generated according to the SBM.
}
\label{figure1}
\end{figure}

In all above experiments, we set $t = 2$; this implies that in cases with $K \geq 3$, the LP relaxation can be further strengthened
by adding inequalities of the form~\eqref{eq3k} corresponding to subsets $S$ with $|S| > 2$.
However, such a strengthening implies a significantly higher cost for the LP relaxation and requires the implementation of a clever separation algorithm. We now show that even for $K > 3$, the weakest LP relaxation, \ie Problem~\eqref{lp:pK} with $t=2$, outperforms the SDP relaxation. To this end, we set $m = 2$ and let $K \in \{4,5\}$; the points in each cluster are generated according to the SBM and the cluster centers have a hive-shaped geometry (see Figure~\ref{figure2}). As before, we set $\Delta \in [2:0.01: \bar \Delta]$ and for each fixed $\Delta$, we conduct 20 random trials. Our results are depicted in Figure~\ref{figure3}. As can be seen from these graphs, the weakest LP relaxation outperforms the SDP relaxation for both $K= 4$ and $K = 5$ clusters.

\begin{figure}[htbp]
 \centering
 \epsfig{figure=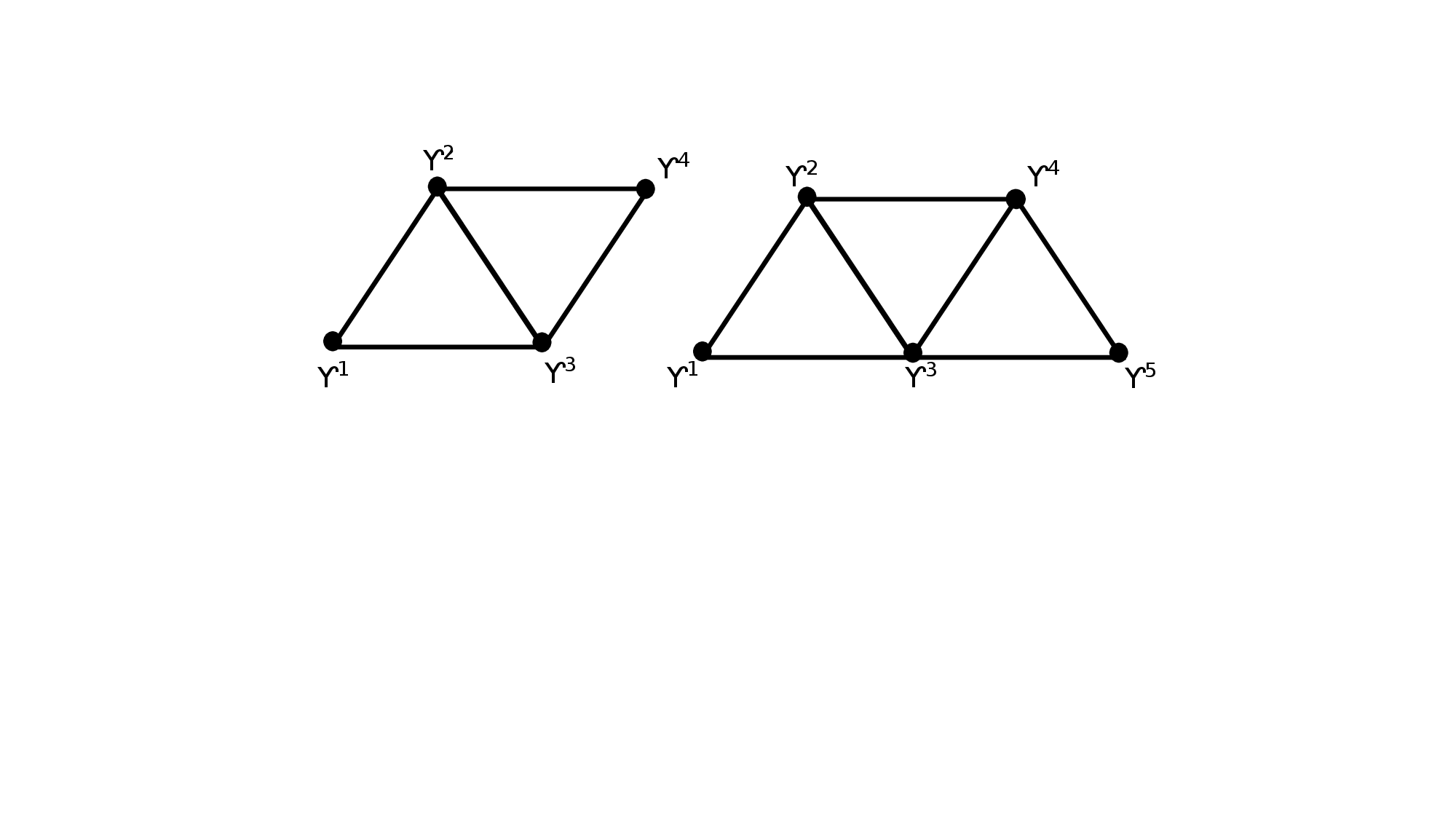, scale=0.25, trim=10mm 100mm 10mm 30mm, clip}
 \caption{Cluster centers $\gamma^k, k\in [K]$ with hive-shaped geometry. The parameter $\Delta$ is defined as the distance between two adjacent centers.}
\label{figure2}
\end{figure}

\begin{figure}[htbp]
 \centering
 \subfigure [$K=4$, $n = 100$]{\label{fig3a}\epsfig{figure=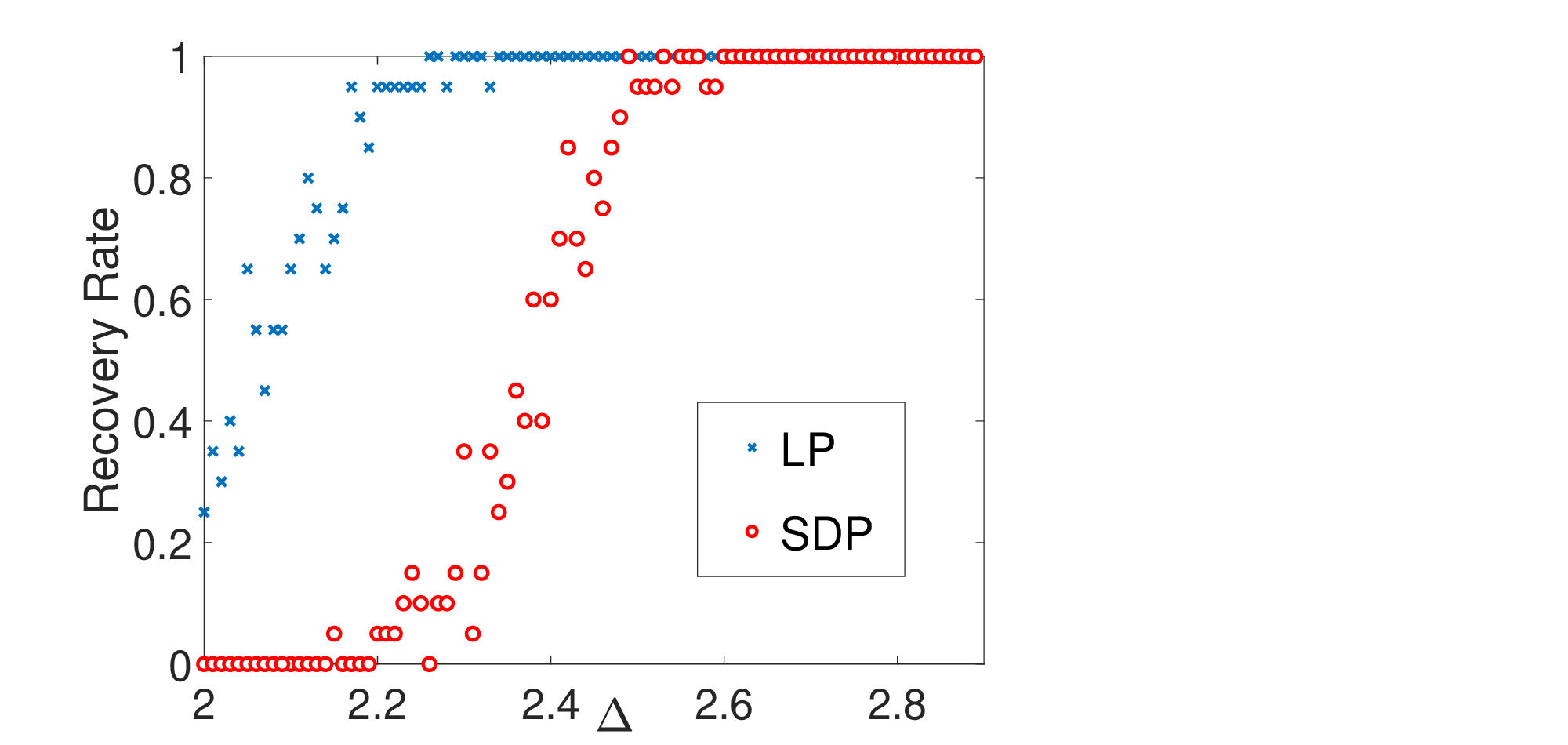, scale=0.27, trim=10mm 0mm 100mm 0mm, clip}}
 \subfigure[$K=5$, $n = 125$]{\label{fig3b}\epsfig{figure=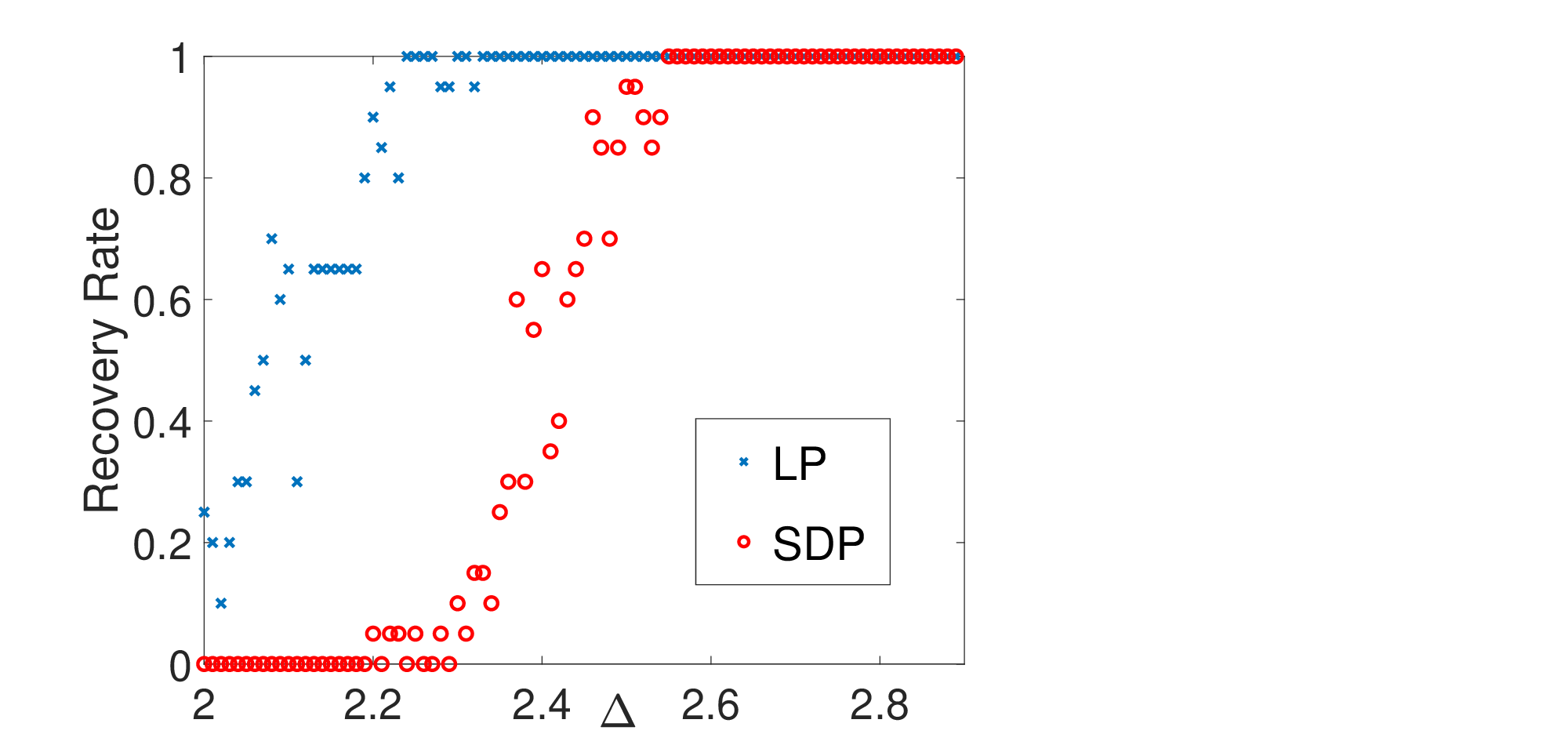, scale=0.27, trim=10mm 0mm 100mm 0mm, clip}}
 \caption{The empirical probability of success of the
weakest LP versus the SDP in recovering planted clusters for the SBM in dimension two, with $K =4$ clusters (Figure~\ref{fig3a})
and with $K =5$ clusters (Figure~\ref{fig3b}).}
\label{figure3}
\end{figure}

We now illustrate the impact of inequalities~\eqref{eq3k} on the quality of the LP relaxation via a simple numerical experiment. As we mentioned before, a careful selection of parameter $t$ requires the development of a separation algorithm that is beyond the scope of this paper. We consider two LP relaxations of K-means clustering: (i) Problem~\eqref{lp:pK} with $t = 2$, referred to as LPt2 and (ii) Problem~\eqref{lp:pK} with $t = 3$, referred to as LPt3. As before, we set $m = 2$ and we generate the points according to the SBM. We consider $K \in \{4,5\}$ where the cluster centers are chosen as shown in Figure~\ref{figure2}. To better understand the impact of inequalities~\eqref{eq3k}, in addition to recovery, we compare the~\emph{tightness} of the two LPs. That is, we say that a relaxation is tight, if the returned optimal solution is a partition matrix as defined by~\eqref{pm}. Moreover, we assume that the balls from which the points are drawn may overlap; that is, we let $\Delta \in [0:0.1: \bar \Delta]$ and for each fixed $\Delta$ we conduct 50 random trials. Indeed in practice, there is often no clear separation between the underlying clusters, and the recovery question does not make much sense; in such settings, one is interested in finding an optimal clustering (\ie an optimal partition matrix) which may or may not correspond to a planted clustering.

Our results are depicted in Figure~\ref{figure4}. Interestingly, while LPt2 and LPt3 have identical performance with respect to recovery, the addition of inequalities~\eqref{eq3k} with $|S| =3$, clearly improves the tightness of the LP relaxation. This serves as a strong motivation for developing an efficient separation algorithm for inequalities~\eqref{eq3k}. We conclude by acknowledging that in order to fully investigate the relative computational benefits of the LP relaxation versus the SDP relaxation for K-means clustering, a comprehensive numerical study on various real data sets is needed. This is indeed a subject of future research.

\begin{figure}[htbp]
 \centering
 \subfigure [$K=4$, $n = 40$]{\label{fig4a}\epsfig{figure=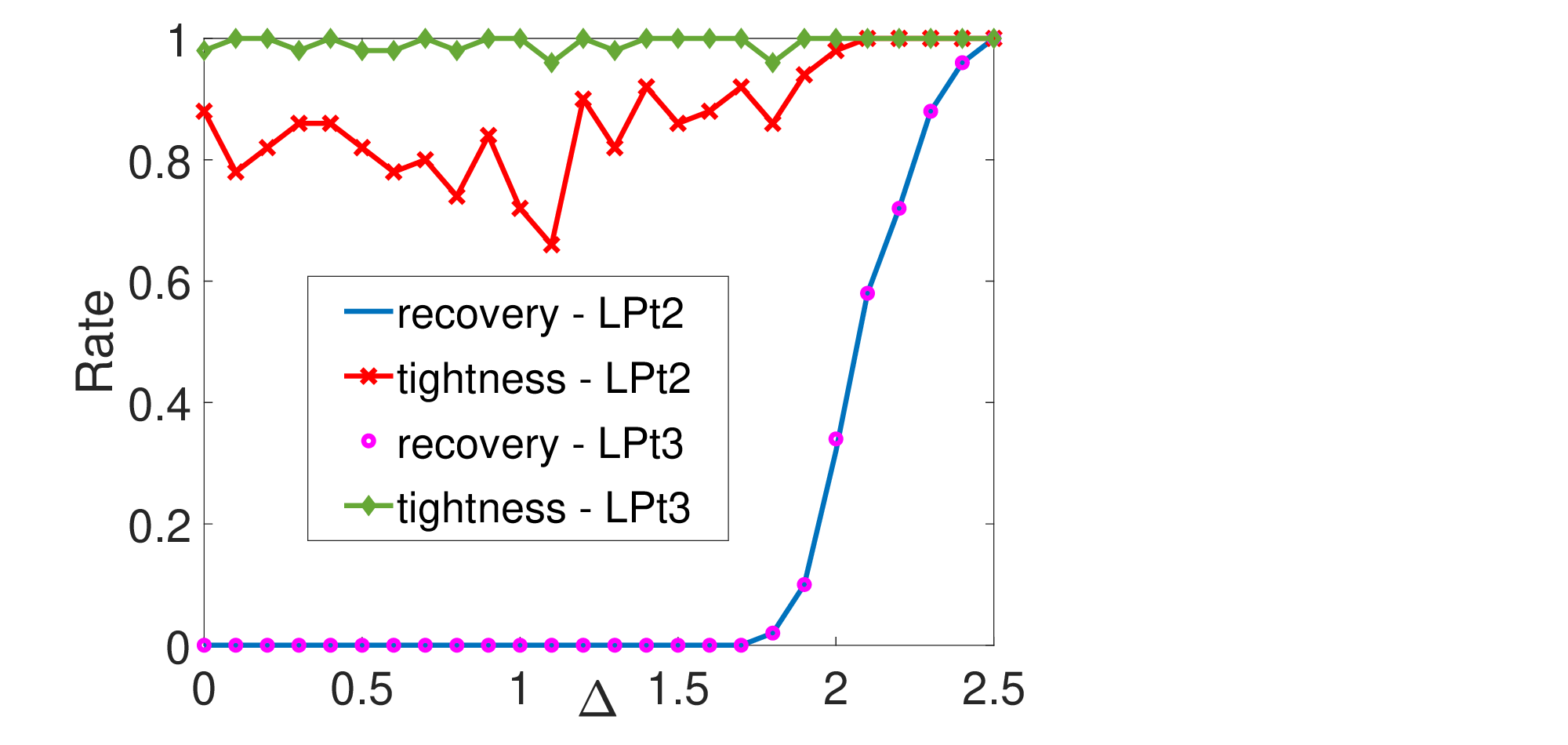, scale=0.28, trim=10mm 0mm 100mm 0mm, clip}}
 \subfigure[$K=5$, $n = 50$]{\label{fig4b}\epsfig{figure=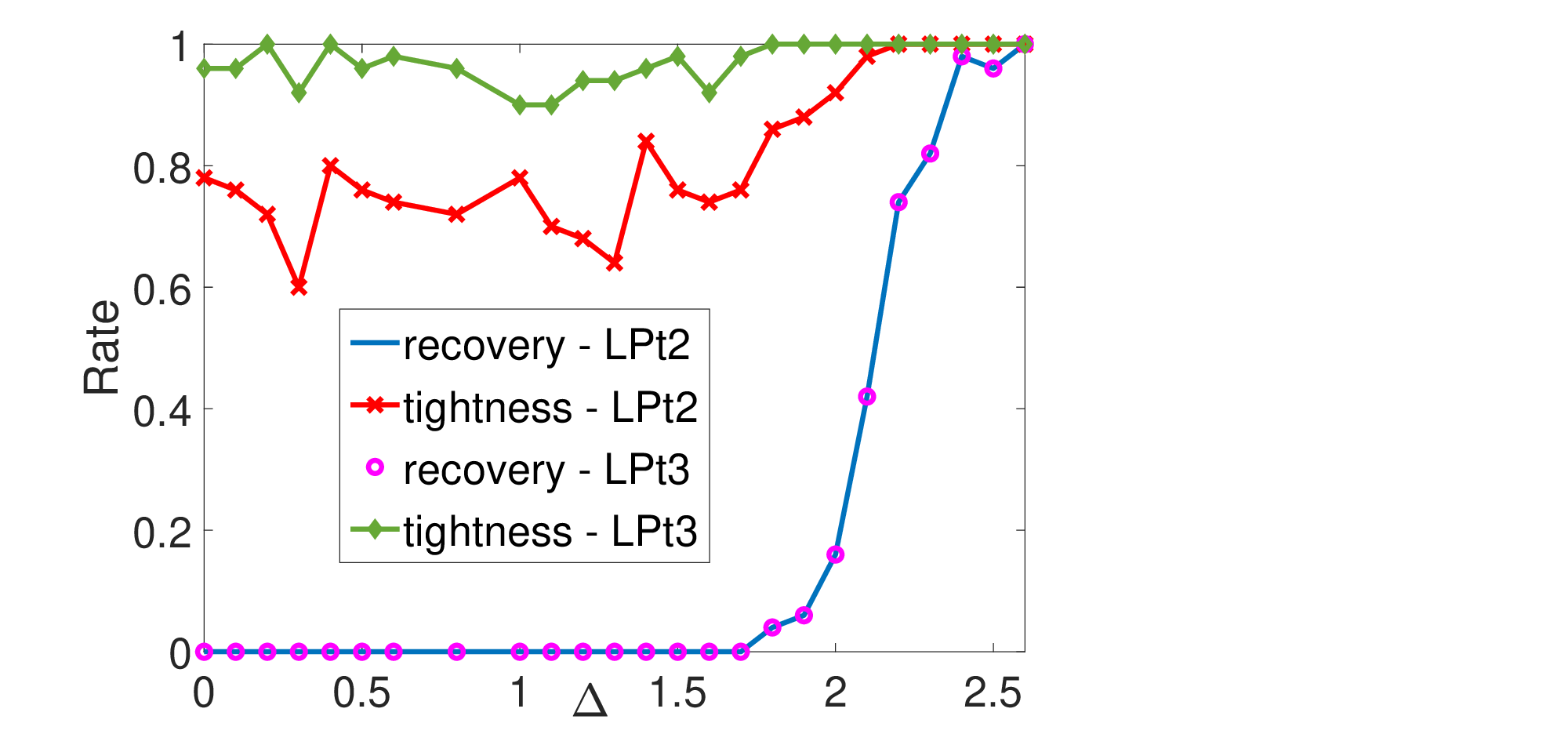, scale=0.28, trim=10mm 0mm 100mm 0mm, clip}}
 \caption{Comparing the quality of two LP relaxations, LPt2 vs LPt3 for the SBM in dimension two, with $K =4$ clusters (Figure~\ref{fig4a})
and with $K =5$ clusters (Figure~\ref{fig4b}).
}
\label{figure4}
\end{figure}

\section{Technical proofs}
\label{sec:proofs}
In this section, we present the two technical lemmas that we utilized to prove Theorem~\ref{recovery}. In the following,
for a Borel set $A\subset \R^m$ and measurable function $f:\R^m\to \R$, we define
$$\dashint_{A} f(x) d\H^{m}(x):=\frac 1{\H^m(A)}\int_{A} f(x) d\H^{m}(x).$$
Moreover, for any $x\in \R^m$, we denote by $x_i$ the $i$th component of $x$. Given two points $x,y\in \R^m$, the notation $x\parallel y$ means that $x$ and $y$ are linearly dependent.

\begin{lemma}\label{MasterOfProbability}
Suppose that the random points are generated according to the SBM.
Then the following inequality holds provided that $\Delta > 1+ \sqrt{3}$:
\begin{equation}\label{2222}
\epsilon:=\frac 13 \Big(\inf_{i,j \in \C_1}\avg^k\Big[\underset{k \in \C_2}{\dashsum}{\min\{d_{ik} + \avg_j[ \din_j] , d_{jk} + \avg_i[\din_i]\}}\Big]- d_{ij}-\avg[\eta] \Big)>0.
\end{equation}
\end{lemma}

\begin{proof}
Denote by $\B_1$ and $\B_2$ the balls corresponding to the first and second clusters, respectively.
Up to a rotation we can assume that the center of $\B_1$ and  $\B_2$  are $0$ and $\Delta e_1$, respectively.
For notational simplicity, we denote the $i$th (resp. $j$th) point in $\B_1$ by $x$ (resp. $y$).
By Lemma \ref{l:condition}, it suffices to show that inequality~\eqref{2222} can be equivalently written as:
\begin{equation}\label{0000}
\max_{x,y \in \B_1}\dashint_{\B_2} \max\{x^Tz, y^Tz\} d\H^{m}(z) -x^Ty< \frac 12\Delta^2.
\end{equation}
First, notice that for any $i \in \B_1$ we have
\begin{equation}\label{0}
\begin{split}
\avg_i[\din_i]&=\dashint_{\B_1}\|x-z\|^2 d\H^{m}(z)= \|x\|^2+\dashint_{\B_1}\|z\|^2d\H^{m}(z)-2x^T\dashint_{\B_1}zd\H^{m}(z)\\
& =\|x\|^2+\dashint_{\B_1}\|z\|^2d\H^{m}(z).
\end{split}
\end{equation}
By symmetry, the same calculation holds for $\avg_i[\din_i]$ with $i \in \B_2$.
By~\eqref{0}, we have
\begin{equation}\label{p}
\begin{split}
\avg\big[\underset{k \in \B_1}{\dashsum}{\din_k}\big]&=\avg\big[\underset{k \in \B_2}{\dashsum}{\din_k}\big]\\
&=\dashint_{\B_1}\left(\|z\|^2+\dashint_{\B_1}\|w\|^2d\H^{m}(w) \right)d\H^{m}(z)=2\dashint_{\B_1}\|z\|^2d\H^{m}(z).
\end{split}
\end{equation}
Hence, by \eqref{0} and \eqref{p}, inequality~\eqref{2222} reads
$$
\min_{x,y \in \B_1}\dashint_{\B_2} \min\{\|x-z\|^2+\|y\|^2, \|y-z\|^2+\|x\|^2\} d\H^{m}(z) - \|x-y\|^2> \dashint_{\B_1}\|z\|^2d\H^{m}(z),
$$
which expanding the squares gives
\begin{equation}\label{000}
\min_{x,y \in \B_1}\dashint_{\B_2}\|z\|^2+ \min\{-2x^Tz, -2y^Tz\} d\H^{m}(z) +2x^Ty> \dashint_{\B_1}\|z\|^2d\H^{m}(z).
\end{equation}
Via a change of variables
\begin{equation*}
\begin{split}
\dashint_{\B_2}\|z\|^2&d\H^{m}(z) = \dashint_{\B_1}\|\Delta e_1 +z\|^2d\H^{m}(z)=\dashint_{\B_1}\|\Delta e_1\|^2+\|z\|^2+2\Delta e_1^Tzd\H^{m}(z)\\
&=\Delta^2+\dashint_{\B_1}\|z\|^2d\H^{m}(z)+2\Delta e_1^T\dashint_{\B_1}zd\H^{m}(z)=\Delta^2 +\dashint_{\B_1}\|z\|^2d\H^{m}(z),
\end{split}
\end{equation*}
hence \eqref{000} reads
\begin{equation*}
\min_{x,y \in \B_1}\dashint_{\B_2} \min\{-2x^Tz, -2y^Tz\} d\H^{m}(z) +2x^Ty> -\Delta^2,
\end{equation*}
which is equivalent to \eqref{0000}.
\end{proof}

\begin{lemma}\label{l:condition}
Inequality \eqref{0000} holds if and only if $\Delta > 1+ \sqrt{3}$.
\end{lemma}
\begin{proof}
We will prove that the maximum of the left-hand side of inequality~\eqref{0000} over all $x, y \in \B_1$ is attained at $(e_1,-e_1)$.
This in turn implies that inequality~\eqref{0000} is satisfied if and only if
$$
\Delta+1=\dashint_{\B_2} z_1 d\H^{m}(z) +1< \frac 12\Delta^2,
$$
which is true if and only if $\Delta > 1 + \sqrt{3}$; \ie the desired condition.

Define
$$F(x,y):=\dashint_{\B_2}\max\{x^Tz,y^Tz\} -x^Ty d\H^{m}(z).$$
Our goal is to show that
\begin{equation}\label{p10}
 \max_{x,y\in \B_1} F(x,y) = F(e_1,-e_1).
\end{equation}
We divide the proof in several steps:

\smallskip

\noindent
{\em {\bf Step 1}. Slicing:

Let $z, w$ be any pair of points in $\B_2$ satisfying $z_1=w_1$, $z_2=-w_2\geq 0$,  $z_j=w_j=0$ for all $j \in \{3,\ldots,m\}$. In the special case $m=1$, we consider $z=w\in \B_2$.
 Define
$$
H(x,y):=\left\{ \frac 12 \max\{x^Tz,y^Tz\} + \frac 12 \max\{x^Tw,y^Tw\} -x^Ty \right\}.
$$
Then \eqref{p10} holds if the following holds
\begin{equation}\label{p20}
\max_{x,y\in \B_1}H(x,y)=H(e_1,-e_1).
\end{equation}
}
{\em Proof of Step 1.} Since
$$F(x,y)=\frac 1{\H^{m}(\B_2)}\int_{\Delta-1}^{\Delta+1} \int_{\{z_1=s\}\cap \B_2}\max\{x^Tz,y^Tz\} -x^Ty d\H^{m-1}(z)ds,$$
to show \eqref{p10} it is enough to show that the function
$$G(x,y):= \int_{\{z_1=s\}\cap \B_2}\max\{x^Tz,y^Tz\} -x^Ty d\H^{m-1}(z)$$
is maximized in $x=e_1$, $y=-e_1$, for every $s \in [\Delta-1,\Delta+1].$
Denoting $A:=\{z_1=s,\, z_2\geq 0\}\cap \B_2$, then
$$\frac 12G(x,y)= \int_{A}\frac 12\max\{x^Tz,y^Tz\}  + \frac 12\max\{x^T(2se_1-z),y^T(2se_1-z)\} -x^Ty d\H^{m-1}(z).$$
Hence, it is enough to prove that for every $s \in [\Delta-1,\Delta+1]$ and for every $z\in A$,
\begin{equation}\label{p30}
\max_{x,y\in \B_1}\left \{\frac 12 \max\{x^Tz,y^Tz\} + \frac 12 \max\{x^T(2se_1-z),y^T(2se_1-z)\} -x^Ty \right\},
 \end{equation}
 is achieved at $(e_1,-e_1)$.
 Since Problem \eqref{p30} is invariant under a rotation of the space around the axis generated by $e_1$,  we conclude that solving Problem \eqref{p30} is equivalent to solving Problem \eqref{p20}.

\noindent
{\em {\bf Step 2}. Symmetric distribution of the maxima:

Let $z,w$ be any pair of points as defined in Step 1.
Define
$$I(x,y):=\frac 12 x^Tz +\frac 12 y^Tw -x^Ty.$$
In order to show that~\eqref{p20} holds, it suffices to prove that for $m\geq 2$
\begin{equation}\label{p40}
\mathop{\max_{x,y\in \B_1}}_{x^Tz\geq y^Tz,\, y^Tw \geq x^Tw }\{I(x,y) \} \leq H(e_1,-e_1)=z_1+1.
\end{equation}
}
{\em Proof of Step 2.}  We first analyze the special case $m=1$. For every $z\in \B_2$, we have that $H(x,y)=\max\{xz,yz\} -xy$. The maximum for $x,y\in [-1,1]$ of the multilinear function $H$ must be attained at the boundary point $x=1$, $y=-1$. This complesetes the proof of Lemma \ref{l:condition} for $m=1$.

From now on, we focus our attention on the case $m\geq 2$.
Assume by contradiction that~\eqref{p40} holds, but~\eqref{p20} does not hold. Then there exists $z,w\in \B_2$, $z_1=w_1$, $z_2=-w_2$ and $z_j=w_j=0$ for all $j \in \{3,\dots,m\}$ and $\bar x,\bar y\in \B_1$ such that $\bar x^T z\geq \bar y^Tz,  \bar x^Tw  \geq \bar y^Tw$ and $H(\bar x,\bar y)>H(e_1,-e_1)$. We deduce that
\begin{equation}\label{minore0}
\begin{split}
H(e_1,-e_1)&<H(\bar x,\bar y)=\frac 12 \bar x^Tz + \frac 12 \bar x^Tw -\bar x^T\bar y\leq \max_{x,y\in \B_1}\frac 12  x^Tz + \frac 12  x^Tw -x^Ty\\
&=\max_{x\in \B_1}\frac 12 x^Tz +\frac 12 x^Tw +\|x\|.
\end{split}
\end{equation}
Since the function $\frac 12 x^Tz +\frac 12 x^Tw +\|x\|$ is convex in $x$, then
\begin{equation}\label{minore1}
\max_{x\in \B_1}\frac 12 x^Tz +\frac 12 x^Tw +\|x\|=\max_{x\in \partial \B_1}\frac 12 x^Tz +\frac 12 x^Tw +\|x\|=\max_{x\in \partial \B_1}\frac 12 x^Tz +\frac 12 x^Tw +1.
\end{equation}
The maximization problem on the right hand side of \eqref{minore1}
has critical points satisfying
\begin{equation}\label{ELE}
z + w +2\lambda x=0.
\end{equation}
Since $z,w\in \B_2$, $z_1=w_1$, $z_2=-w_2\geq 0$ and $z_j=w_j=0$ for every $j=3,\dots,m$, then equation \eqref{ELE}  implies that $x\parallel e_1$.
We deduce that any maximum point of \eqref{minore1} satisfies $x=te_1$, with $t\in[-1,1]$:
\begin{equation}\label{minore2}
\max_{x\in \partial \B_1}\frac 12 x^Tz +\frac 12 x^Tw +1=\max_{t\in[-1,1]}tz_1 +1,
\end{equation}
and the maximum point of~\eqref{minore2} is attained at $t=1$, since $z_1>0$.
Combining \eqref{minore0}, \eqref{minore1} and \eqref{minore2}, we deduce the contradiction $H(e_1,-e_1)<z_1+1=H(e_1,-e_1)$.

\noindent
{\em {\bf Step 3}. Reduction from balls to disks:

To show the validity of~\eqref{p40}, we can restrict to dimension $m=2$.}

{\em Proof of Step 3.}  We fix $z=(z_1,z_2,0,\dots,0)$ and $w=(z_1,-z_2,0,\dots,0)$.
Denote $x=(x_1,x')$ and $y=(y_1,y')$, where $x':=(x_2,\dots,x_m)$ and $y':=(y_2,\dots,y_m)$. We will use the same notation also for $z,w$.
Moreover denote $\tilde x=(x_1, \tilde x')$ and $\tilde y=(y_1, \tilde y')$, where $\tilde x':= (\tilde x_2,0,\dots,0)$, $\tilde y':= (\tilde y_2,0,\dots,0)$, such that
$\|x'\|=\|\tilde x'\|$ and $\|y'\|=\|\tilde y'\|$, $\tilde x_2\geq 0$ and $\tilde y_2\leq 0$.
To prove the claim, it suffices to show that the maximum in~\eqref{p40} is attained at $x,y \in \mbox{span}\{ e_1, e_2 \}$. To this end, it is enough to show that $I(x,y)  \leq I(\tilde x,\tilde y)$,
which is equivalent to
$$\frac 12 (x')^Tz' +\frac 12 (y')^Tw' -(x')^Ty'  \leq \frac 12 (\tilde x')^Tz' + \frac 12 (\tilde y')^Tw' -(\tilde x')^T\tilde y'.$$
In turn, this inequality is valid because
\begin{itemize}
\item [(i)] by definition $\tilde x' \parallel  z'$, $\tilde y' \parallel  w'$, $\tilde x_2\geq 0$, $\tilde y_2\leq 0$ and $z_2\geq 0$; then  we have $(x')^Tz'\leq (\tilde x')^Tz'$ and $(y')^Tw' \leq (\tilde y')^Tw'.$
\item [(ii)] by definition $\|x'\|=\|\tilde x'\|$, $\|y'\|=\|\tilde y'\|$, $\tilde x_2\geq 0$, $\tilde y_2\leq 0$, and  $\tilde x' \parallel  \tilde y'$; then we have $(x')^Ty' \geq (\tilde x')^T\tilde y'.$
\end{itemize}

\noindent
{\em {\bf Step 4}. Reduction from disks to circles:

In order to prove Problem \eqref{p40}, it is enough to show that
\begin{equation}\label{p50}
\mathop{\max_{x,y\in \partial \B_1}}_{x^Tz\geq y^Tz,\, y^Tw \geq x^Tw }I(x,y) \leq H(e_1,-e_1)=z_1+1,
\end{equation}
for every $z\in \B_2$, $z_1=w_1$ and $z_2=-w_2\geq 0$.
}

{\em Proof of Step 4.}  Fix $x,y\in \B_1$. To prove this step, it is enough to find $x',y'\in \partial \B_1$ such that $I(x',y')\geq I(x,y)$.
Let us denote $\bar x= x/\|x\|$ if $x\neq 0$ and $\bar x=e_1$ if $x=0$. Let us denote $\bar y= y/\|y\|$ if $y\neq 0$ and $\bar y=e_1$ if $y=0$. Since $x,y,z,w$ are fixed, we define the constants $a=\frac 12 \bar x^Tz$, $b=\frac 12 \bar y^Tw$ and $c=\bar x^T \bar y$.
We consider the problem
\begin{equation*}
\mathop{\max_{(r_1,r_2)\in [-1,1]^2}}I(r_1\bar x,r_2 \bar y)=\mathop{\max_{(r_1,r_2)\in [-1,1]^2}}r_1a+r_2b-r_1r_2c.
\end{equation*}
It is well-known that the maximum of a bilinear function over a box is attained at a vertex of the box and this completes the proof.

\noindent
{\em {\bf Step 5}. Symmetric local maxima:

For any pair $x,y\in \partial \B_1$ of the form $x_1 = y_1$ and $x_2 = -y_2$, we have
$$I(x,y)\leq H(e_1,-e_1).$$}
{\em Proof of Step 5.}
Given such symmetric pair $(x,y)$, the objective function evaluates to
$I(x,y)= z_1 x_1 + z_2 x_2 -x_1^2 + x_2^2$.  Using $x^2_1 + x^2_2 = 1$ and  $z_2 \sqrt{1-x^2_1} \leq z_2$, it suffices to show that
\begin{equation}\label{eqA1}
z_2 \leq  2 x_1^2 - z_1 x_1 + z_1, \quad \forall  x_1 \in [-1,1].
\end{equation}
Since the function $\hat f(x_1):=2 x_1^2 - z_1 x_1 + z_1$ on the right hand side of~\eqref{eqA1} is a convex parabola in $x_1$,
its minimum is either attained at one of the end points or at $\tilde x_1 = \frac{z_1}{4}$, provided
that $-1 \leq \frac{z_1}{4} \leq 1$. Since $\Delta -1 \leq z_1 \leq \Delta + 1$, the point $\tilde x_1$ lies in the domain only if
$\Delta-1 \leq z_1 \leq \min\{4, \Delta + 1\}$.
The value of $\hat f$ at $x_1 = -1$ and $x_1 = 1$ evaluates to $2 + 2 z_1$ and $2$, respectively,
both of which are bigger than $z_2$. Hence it remains to show that
$$
\sqrt{1- u^2} \leq (u + \Delta) - \frac{(u+\Delta)^2}{8}, \quad -1 \leq u \leq \min\{4-\Delta, 1\},
$$
where we set $u := z_1 - \Delta$ and we use that $z_2\leq \sqrt{1-u^2}$. Since $u + \Delta \leq 4$, the right hand side of the above inequality is increasing in $\Delta$;
hence it suffices to show its validity at $\Delta = 2$; \ie
$$
\sqrt{1- u^2} \leq (u + 2) - \frac{(u + 2)^2}{8}, \quad -1 \leq u \leq 1.
$$
The right-hand side of the above inequality is concave and hence is lower bounded by its secant line through the boundary points $(-1,7/8)$, $(1,15/8)$; hence
it suffices to show that
$\sqrt{1- u^2} \leq \frac{1}{2} (u+ \frac{11}{4})$ for all $-1 \leq u \leq 1$.
Squaring both sides and rearranging the terms, the above inequality can be equivalently written as
$u^2 + \frac{11}{10} u + \frac{57}{80} \geq 0$, where $-1 \leq u \leq 1$. It can be shown that the minimum of the left hand side
of this inequality is attained at $u = -\frac{11}{20}$ and is equal to $0.41$ and this completes the proof.

\noindent
{\em {\bf Step 6}. Decomposition of the circle:

To solve Problem~\eqref{p50}, it suffices to solve
\begin{equation}\label{p6}
\mathop{\max_{x,y\in  \partial \B_1\cap \{x_1\leq 0,\, y_1 \geq 0\}}}_{x^Tz\geq y^Tz, \, y^Tw \geq x^Tw }I(x,y) \leq H(e_1,-e_1)=z_1+1,
\end{equation}}
for every $z\in \B_2$, $z_1=w_1$, $z_2=-w_2\geq 0$.

{\em Proof of Step 6.}  We first consider the case when $x_1\leq 0$ and $y_1\leq 0$. Since $z_1>0$, then $x_1z_1\leq -x_1z_1$, $y_1z_1\leq -y_1z_1$. We deduce that $I(x,y)\leq I((-x_1,x_2),(-y_1,y_2))$.

Now, since the case $\{x_1\geq 0,y_1 \leq 0\}$ is symmetric to the case $ \{x_1\leq 0,y_1 \geq 0\}$, we just need to show that
$$\mathop{\max_{x,y\in  \partial \B_1\cap \{x_1\geq 0,\, y_1 \geq 0\}}}_{x^Tz\geq y^Tz, \, y^Tw \geq x^Tw }I(x,y) \leq H(e_1,-e_1)=z_1+1.$$
Consider $x,y\in  \partial \B_1$ such that $x_1\geq 0$, $y_1 \geq 0$, $x^Tz\geq y^Tz$, $y^Tw \geq x^Tw$. If $x_2$, $y_2$ are both negative (resp. both positive), we can consider the new couple $(x_1,-x_2)$, $(y_1,y_2)$ (resp. $(x_1,x_2)$, $(y_1,-y_2)$), which gives a bigger (or equal) value for $I$. Hence, we can restrict our study to the case $x_2\geq 0$ and $y_2\leq 0$. We denote in spherical coordinates $x=(1,\theta)$, $y=(1,\eta)$, $z=(\|z\|,\gamma)$  and $w=(\|z\|,2\pi- \gamma)$. Since $x,y\in  \partial \B_1$, $x_1\geq 0$, $y_1 \geq 0$, $x_2\geq 0$ and $y_2\leq 0$ then $\theta \in [0, \pi/2]$ and $\eta\in [3\pi/2,2\pi]$. Furthermore, since $z,w \in \B_2$, we can easily verify that
\begin{equation}\label{cond2}
\gamma\in [0,\pi/4),
\end{equation}
since the straight line parallel to $e_1+e_2$ does not intersect $\B_2$, for every $\Delta \geq 2$.

With this notation we have
$$I(\theta, \eta)=\frac 12 \|z\| \cos (\theta -\gamma)+\frac 12 \|z\| \cos (\eta+\gamma)-\cos(\eta-\theta).$$
It then follows that the critical points of the above function have to satisfy
$$\frac 12 \|z\| \sin (\theta -\gamma)+\sin(\eta-\theta)=0, \quad \mbox{and} \quad -\frac 12 \|z\| \sin (\eta+\gamma)+\sin(\eta-\theta)=0.$$
Subtracting the two equations, since $\|z\|>0$, we deduce that
\begin{equation}\label{crit}
\sin (\theta -\gamma)=-\sin (\eta+\gamma).
\end{equation}
We observe that $\theta -\gamma\in  [-\pi/2,\pi/2]$ and $\eta+\gamma \in [3\pi/2,5\pi/2]$.
Since $\sin(s)$ is injective for $s\in [-\pi/2,\pi/2]$, we deduce that the only solution is $\theta=2\pi-\eta$. This critical point corresponds to a symmetric  couple $x_1=y_1,x_2=-y_2$ and by Step 5 we have $I(x,y)\leq H(e_1,-e_1)$.
Up to rotation, the boundary cases of $\theta$ and $\eta$ coincide. Hence, we are just left to study the boundary case $\theta=0$, or equivalently $(x,y)=(e_1,y)$ (the boundary case $\theta=\pi/2$ gives $x_1\leq 0$ and will be threated in Step 7). In this case, since $y_1\geq 0$ and $ y_1,y_2,z_2\in [-1,1]$
$$I(e_1,y)=\frac 12 z_1+ \frac 12 y_1z_1 -  \frac 12 y_2z_2 -y_1\leq  \frac 12 z_1(1+y_1) +1/2 < z_1+1 =H(e_1,-e_1).$$

\noindent
{\em {\bf Step 7}. We solve Problem \eqref{p6}.}

{\em Proof of Step 7.}  We now assume that $x,y\in  \partial \B_1$, $x_1\leq 0,y_1 \geq 0$, $x^Tz\geq y^Tz$ and $y^Tw \geq x^Tw$.  As explained in Step 6, we can also assume that $x_2 \geq 0$ and $y_2 \leq 0$. This implies that, using the notation of Step 6, we need to study the domain
\begin{equation}\label{condp}
(\theta,\eta) \in [\pi/2, \pi/2+2\gamma]\times [3\pi/2 ,2 \pi].
\end{equation}
Using a similar line of argument as in Step 6, all critical points in this region must satisfy \eqref{crit}.
By \eqref{cond2}, since the function $\sin(s)$ is strictly increasing in $[-\gamma, \pi/2-\gamma]$ and $\sin(s)> \sin(\pi/2-\gamma)=\sin (\pi/2+\gamma)$ for every $s\in (\pi/2-\gamma, \pi/2+\gamma)$, it follows that $\sin ([\pi/2-\gamma, \pi/2+\gamma])\cap\sin ([-\gamma, \pi/2-\gamma])=\{\sin( \pi/2-\gamma)\}$ and that the equation \eqref{crit} is never satisfied in the interior of the region \eqref{condp}.
This implies that the only maximum points can be achieved at the boundary.
We are just left to check that for all the boundary points $(x,y)$ of this region, $I(x,y)\leq H(e_1,-e_1)$.

We start with $\theta=\pi/2$, that is, all the points of the form $(e_2,y)$ with $y\in  \partial \B_1$. We claim that
$$I(e_2,y)=\frac 12z_2+\frac 12 y_1z_1 -\frac 12y_2z_2-y_2 \leq z_1 + 1, \quad \forall y\in  \partial \B_1.$$
The maximum of the linear function $I(e_2,y)$ over $y\in  \partial \B_1$ is attained at
$$\tilde y=\frac{(z_1,-(z_2+2))}{\sqrt{z_1^2+(z_2+2)^2}}.$$
Hence, it suffices to show that the following inequality is valid:
$$
I(e_2,\tilde y) = \frac 12z_2+\frac12\sqrt{z_1^2+(z_2+2)^2} \leq z_1 +1, \quad  \Delta - 1 \leq z_1 \leq \Delta +1, \; 0 \leq z_2 \leq 1,
$$
Defining $u = z_1 - \Delta$, the above inequality can be equivalently written as:
\begin{equation}\label{a1}
\sqrt{1- u^2} \leq (u+\Delta) - \frac{(u+\Delta)^2}{4(u+\Delta+2)}, \quad -1 \leq u \leq 1,
\end{equation}
First, notice that the right hand side of inequality~\eqref{a1} is increasing in $\Delta$, hence it suffices to show its validity
at $\Delta = 2$. Second this expression is concave is and hence can be lower bounded by its secant line, denoted by $ a u + b$; therefore, it suffices to show that $\sqrt{1- u^2} \leq a u + b$. Squaring both sides, we need to show that $(a u + b)^2 + u^2 \geq 1$ for $-1 \leq u \leq 1$ and it can be checked that the latter inequality is valid.

The calculations for the boundary case $\eta=3\pi/2$, that is all the points $(x,-e_2)$ with $x\in  \partial \B_1$, are symmetric to the case $\theta=\pi/2$ (up to a rotation).

We now consider the boundary case $\eta=2\pi$, that is all the points $(x,e_1)$ with $x\in  \partial \B_1$ and we claim that $I(x,e_1)\leq z_1+1$ for every $x\in  \partial \B_1$. Indeed, since $z_1\geq 1$
$$I(x,e_1)=\frac 12x_1z_1+\frac 12 x_2z_2 +\frac 12z_1-x_1=\left(\frac 12-\frac 1{z_1}\right)x_1z_1+\frac 12 x_2z_2 +\frac 12z_1, \qquad \forall x\in  \partial \B_1.$$
Since $\frac 1{z_1}\in (0,1]$, we have $\left(\frac 12-\frac 1{z_1}\right)x_1z_1\leq \frac 12z_1$ and since
$x_2,z_2\in [-1,1]$, we have
$$I(x,e_1)\leq \frac 12z_1+\frac 12 x_2z_2 +\frac 12z_1\leq  z_1+1= H(e_1,-e_1), \qquad \forall x\in  \partial \B_1.$$

The last boundary case is $\theta=\pi/2+2\gamma$, for every $\eta\in [3\pi/2,2\pi]$.
In this case, we observe that $\theta-\gamma\geq 2\pi-\eta +\gamma$, where $\theta-\gamma$ is the angle between $x$ and $z$
and $2\pi-\eta +\gamma$ is the angle between $y$ and $z$). Hence $x^Tz\leq y^Tz$ and
$$I(x,y)\leq \frac 12 y^Tz +\frac 12 y^Tw -x^Ty \leq y_1z_1+1\leq z_1+1= H(e_1,-e_1).$$
This concludes the proof of Step 7.
\end{proof}

\section*{Acknowledgments}
The authors would like to thank Shuyang Ling and Soledad Villar for fruitful discussions about the recovery properties of SDP relaxations for clustering problems.

\end{document}